\documentclass[hidelinks,onefignum,onetabnum]{siamart220329}

\usepackage{lipsum}
\usepackage{amsfonts}
\usepackage{graphicx}
\usepackage{epstopdf}
\usepackage{amssymb}
\usepackage{todonotes}
\usepackage{subcaption}
\usepackage{comment}
\ifpdf
  \DeclareGraphicsExtensions{.eps,.pdf,.png,.jpg}
\else
  \DeclareGraphicsExtensions{.eps}
\fi

\usepackage{bm}
\usepackage{xfrac}
\newsiamthm{assumption}{Assumptions}

\newcommand{\I}{\mathbf{I}}

\def\C{{\mathbb C}}
\def\eps{\varepsilon}

\def\wh{\widehat}
\def\wt{\widetilde}

\newcommand\bfr{{\mathbf r}}

\newcommand\bfI{{\mathbf I}}
\newcommand\bfA{{\mathbf A}}

\newcommand\bfF{{\mathbf F}}

\newcommand\bfK{{\mathbf K}}
\newcommand\bfL{{\mathbf L}}
\newcommand\bfM{{\mathbf M}}

\newcommand\bfP{{\mathbf P}}
\newcommand\bfQ{{\mathbf Q}}
\newcommand\bfR{{\mathbf R}}
\newcommand\bfS{{\mathbf S}}

\newcommand\bfU{{\mathbf U}}
\newcommand\bfV{{\mathbf V}}

\newcommand\bfY{{\mathbf Y}}
\newcommand\bfZ{{\mathbf Z}}

\newcommand\bfSigma{{\mathbf \Sigma}}

\def\eps{\varepsilon}

\def\phi{\varphi}

\usepackage{tikz}	
\usepackage{mathabx}

\usepackage{etoolbox}

\csundef{algorithm*}
\usepackage[linesnumbered,ruled,vlined,onelanguage,algo2e]{algorithm2e}

\makeatletter
\newcommand{\ostar}{\mathbin{\mathpalette\make@circled\star}}
\newcommand{\make@circled}[2]{%
  \ooalign{$\m@th#1\smallbigcirc{#1}$\cr\hidewidth$\m@th#1#2$\hidewidth\cr}%
}
\newcommand{\smallbigcirc}[1]{%
  \vcenter{\hbox{\scalebox{0.77778}{$\m@th#1\bigcirc$}}}%
}
\makeatother
\newcommand\mat{\mathbf{Mat}}
\newcommand\ten{\text{Ten}}
\newcommand{\norm}[1]{\| #1  \|}

\newsiamremark{remark}{Remark}
\newsiamremark{hypothesis}{Hypothesis}
\crefname{hypothesis}{Hypothesis}{Hypotheses}
\newsiamthm{claim}{Claim}

\headers{Parallel BUG integrator for TTNs}{G. Ceruti, J. Kusch, Ch. Lubich, D. Sulz}

\title{
A parallel Basis Update and Galerkin Integrator for Tree Tensor Networks
\thanks{\funding{The work of Christian Lubich and Dominik Sulz was funded by the Deutsche Forschungsgemeinschaft (DFG, German Research Foundation) through the Research Unit FOR 5413/1, Grant No. 465199066 and the Research Unit TRR 352, Project-ID 470903074.}}}

\author{Gianluca Ceruti\footnotemark[2] \and Jonas Kusch\footnotemark[3] \and \footnotemark[4] \and Dominik Sulz\footnotemark[4]}

\author{Gianluca Ceruti\thanks{University of Innsbruck, Austria (\email{gianluca.ceruti@uibk.ac.at}).}
\and Jonas Kusch\thanks{Norwegian University of Life Sciences, Scientific Computing, \r{A}s, Norway (\email{jonas.kusch@nmbu.no}).}
\and Christian Lubich\thanks{Mathematisches Institut, Universität Tübingen, Auf der Morgenstelle 10, 72076 Tübingen, Germany (\email{lubich@na.uni-tuebingen.de} and \email{dominik.sulz@uni-tuebingen.de}).} \and Dominik Sulz\footnotemark[4]}

\usepackage{amsopn}

\begin{document}

\maketitle

\begin{abstract}
Computing the numerical solution to high-dimensional tensor differential equations can lead to prohibitive computational costs and memory requirements. To reduce the memory and computational footprint, dynamical low-rank approximation (DLRA) has proven to be a promising approach. DLRA represents the solution as a low-rank tensor factorization and evolves the resulting low-rank factors in time. A central challenge in DLRA is to find time integration schemes that are robust to the arising small singular values. A robust parallel basis update \& Galerkin integrator, which simultaneously evolves all low-rank factors, has recently been derived for matrix differential equations. This work extends the parallel low-rank matrix integrator to Tucker tensors and general tree tensor networks, yielding an algorithm in which all bases and connecting tensors are evolved in parallel over a time step. We formulate the algorithm, provide a robust error bound, and demonstrate the efficiency of the new integrators for problems in quantum many-body physics, uncertainty quantification, and radiative transfer. 
\end{abstract}



\section{Introduction}\label{sec:intro}
Tree tensor networks (TTNs) are a data-sparse hierarchically structured multilinear parametrization of high-order tensors. The parameters are provided by basis matrices at the leaves of the tree and low-order connecting tensors at the inner nodes of the tree. TTNs have been developed independently in chemistry, physics and mathematics; see e.g. \cite{Wang2003,shi2006classical,hackbusch2012tensor}. They have been successfully used to approximate solutions to high-dimensional evolution and optimization problems. For evolution equations, the Dirac--Frenkel time-dependent variational principle (TDVP) yields a coupled system of differential equations for the basis matrices and connecting tensors \cite{Wang2003,lubich2013dynamical}, which needs to be integrated in time numerically.

Standard general-purpose time integration methods are not suitable for TTNs, because their stepsizes are restricted by the smallest singular values of matricizations of the connecting tensors, which are typically tiny. Recently, specific TTN integrators that are robust to small singular values have been devised in \cite{CLW2021} and \cite{CLS2023}. 
They admit convergent error bounds that are independent of small singular values and
related rapid changes in orthonormal factors in the time-continuous TDVP.
These robust TTN integrators extend, in a nonobvious but systematic way, robust integrators for the dynamical low-rank approximation of matrix differential equations developed in \cite{lubich2014projector} --- the projector-splitting integrator --- and in \cite{ceruti2022rank} --- the rank-adaptive basis update \& Galerkin (BUG) integrator, respectively.

In the TTN integrator of \cite{CLW2021}, the multilinear ranks of the connecting tensors (bond dimensions in physical terminology) are fixed, and differential equations at the nodes of the tree for the basis matrices and connecting tensors are solved sequentially, one after the other. The TTN integrator of \cite{CLS2023} determines the ranks adaptively and solves the differential equations at nodes of the same hierarchical level in parallel, but still sequentially from one level to the next from the leaves to the root.

In this paper we propose and analyze a TTN integrator that  
\begin{itemize}
  \item is robust to small singular values,
  \item is rank-adaptive, and
  \item {\it solves all differential equations at the nodes of the tree in parallel\/} within each time step.
\end{itemize}
A sequential transfer between hierarchical levels remains only in the numerical linear algebra (hierarchical orthogonalization). Besides its improved parallelism, the novel TTN integrator does not require an augmented-rank update of the connecting tensors as in~\cite{CLS2023}. In contrast to integrators based on projector-splitting such as \cite{CLW2021}, it does not solve any differential equations backward in time. This is favorable for dissipative problems, where backward propagation steps are potentially problematic.

We derive the algorithm, give a robust first-order error bound and complement the analytical results by numerical experiments for quantum spin systems and for uncertainty quantification in radiative transfer.

This parallel TTN integrator is a nontrivial extension of the parallel low-rank matrix integrator of  \cite{ceruti2023parallel}. The central idea of that integrator is to approximate the numerical solution of the rank-adaptive BUG integrator of \cite{ceruti2022rank} by a first-order approximation consisting only of terms that can be computed in parallel. An extension to second order has recently been given in \cite{kusch2024second}.

Matrix product states (MPS) \cite{perez2006matrix} or synonymously tensor trains \cite{oseledets2011tensor}
are a class of tensor networks that is widely used in quantum physics; see the recent reviews \cite{cirac2021matrix,paeckel2019time} and numerous references therein. MPS can be viewed as TTNs on the tallest tree for a given number of nodes. 
The projector-splitting integrator for MPS \cite{lubich2015time,haegeman2016unifying} has become popular in quantum physics under the misnomer ``the TDVP algorithm'' (note that there are a variety of very different algorithms discretizing the time-continuous TDVP, e.g.~BUG integrators, and not just the projector-splitting integrator).
A parallel TDVP algorithm for MPS has been proposed in \cite{secular2020parallel}. That integrator achieves parallelism in a remarkable way, which is conceptually different from the construction of the parallel integrator considered here. The parallel integrator of \cite{secular2020parallel} is not constructed to have the same robustness to small singular values and related rapidly changing orthonormal factors and cannot be expected to have a robust error bound. The new parallel TTN integrator studied here can in particular be applied to MPS, where it could be used with existing optimized MPS software, which at present is not fully available for general TTNs. On the other hand, TTNs on balanced binary trees (i.e.~binary trees of minimal height) have been observed to require smaller ranks (bond dimensions) and fewer parameters than MPS in simulations of quantum spin systems \cite{CLS2023,sulz2024numerical} and may thus offer advantages over~MPS.

The paper is structured as follows. After the introduction in Section~\ref{sec:intro}, the parallel matrix integrator is reviewed in Section~\ref{sec:recapParallel}. Section~\ref{sec:parallelTucker} presents the extension of this integrator to Tucker tensors, together with the derivation of a robust error bound. Then, Section~\ref{sec:TTN} extends this derivation to tree tensor networks. The efficiency of the presented integrators is demonstrated by numerical experiments in Section~\ref{sec:resultsnum}.

\section{A modified parallel Basis-Update \& Galerkin matrix integrator}\label{sec:recapParallel}
We begin by reviewing a modified version of the parallel BUG integrator of \cite{ceruti2023parallel} for a matrix-valued problem of the form 
\begin{equation} \label{eq:matfulleq}
\dot{\bfA}(t) = \bfF(\bfA(t)), \quad \bfA(t_0) = \bfA_0\, ,
\end{equation}
where $\bfA(t)\in\mathbb{C}^{m \times n}$ and $\bfF : \mathbb{C}^{m\times n} \rightarrow \mathbb{C}^{m\times n}$. 
We compute factorized low-rank approximations at $t_k = t_0+kh$, of varying rank $r_k\ll m,n$, 
$$
\bfY_k = \bfU_k \bfS_k \bfV_k^* \approx \bfA(t_k),
$$
where the left and right basis matrices $\bfU_k \in \C^{m\times r_k}$ and $\bfV_k \in \C^{n\times r_k}$ have orthonormal columns and the coefficient matrix $\bfS_k\in \C^{r_k\times r_k}$ is nonsingular.
We describe one time step of the method, updating from $t_0$ to $t_1$. This procedure is then repeated to advance the solution approximation to $t_2,t_3,$ etc.

Given the factorized rank-$r$ numerical solution $\bfY_0 = \bfU_0 \bfS_0 \bfV_0^{*}$ at time $t_0$, one step of the modified parallel integrator from $t_0$ to $t_1 = t_0 + h$ reads as follows:
	\begin{enumerate}
		\item 
		Compute augmented basis matrices $\wh \bfU\in \C^{m\times 2r}$ and $\wh \bfV\in \C^{n\times 2r}$ as well as the coefficient matrix ${\bfS}(t_1) \in\C^{r \times r}$ (\emph{in parallel}):
		\\[2mm]
    		\textbf{K-step}:
    		From $t=t_0$ to $t_1$, integrate the $m \times r$ matrix differential equation
    		\begin{equation}\label{K-stepmatrix} 
    		\dot{\textbf{K}}(t) = \bfF( \textbf{K}(t) \bfV_0^*) \bfV_0, \qquad \textbf{K}(t_0) = \bfU_0 \bfS_0.
    		\end{equation}
    		Construct $\wh \bfU = (\bfU_0, \widetilde \bfU_1) \in \C^{m\times 2r}$  
      as an orthonormal basis of the range 
    		of the $m\times 2r$ matrix $(\bfU_0, \textbf{K}(t_1))$ (e.g.~by QR decomposition), where $ \widetilde \bfU_1 \in \C^{m\times r}$ 
    		is filled with zero columns if $(\bfU_0, \textbf{K}(t_1))$ has rank less than $2r$.\\
    		Compute the 
            $r\times r$ matrix $$\widetilde \bfS_1^K = h\widetilde\bfU_1^{*}\bfF(\bfY_0)\bfV_0\in\C^{r\times r} \, .$$
		\\
                \textbf{L-step}: 
    		From $t=t_0$ to $t_1$, integrate the $n \times r$ matrix differential equation
    		\begin{equation}\label{L-step} 
    		\dot{\textbf{L}}(t) =\bfF( \bfU_0 \textbf{L}(t)^*)^*  \bfU_0, \qquad \textbf{L}(t_0) = \bfV_0 {\bfS}_0^*. 
    		\end{equation}
    		Construct $\wh \bfV = (\bfV_0, \widetilde \bfV_1) \in \C^{n\times 2r}$ as an orthonormal basis of the range 
    		of the $n\times 2r$ matrix $(\bfV_0, \textbf{L}(t_1))$ (e.g.~by QR decomposition), where $ \widetilde \bfV_1 \in \C^{n\times r}$ 
    		is filled with zero columns if $(\bfV_0, \textbf{L}(t_1))$ has rank less than $2r$.\\
    		Compute the 
            $r\times r$ matrix  $$\widetilde \bfS_1^L = h\bfU_0^{*}\bfF(\bfY_0)\widetilde\bfV_1\in\C^{r\times r} \, .$$
            \\
    		\textbf{S-step}:
    		From $t=t_0$ to $t_1$, integrate the $r \times r$ matrix differential equation
    		\begin{equation}\label{S-step} 
        		\dot{\bfS}(t) =  \bfU_0^* \bfF( \bfU_0 \bfS(t) \bfV_0^*) \bfV_0, 
        		\qquad \bfS(t_0) = \bfS_0.
              \end{equation}
		\item \textbf{Augment}:
		Form the augmented coefficient matrix $\widehat \bfS_1 \in\C^{2r\times 2r}$ as
		\begin{align} \label{wh-S1}
                \widehat \bfS_1 = 
                \begin{pmatrix} 
                     \bfS(t_1) & \widetilde \bfS_1^L \\
                    \widetilde \bfS_1^K & \mathbf{0}
                \end{pmatrix}.
            \end{align}
            \item \textbf{Truncate:} Compute the singular value decomposition $\wh \bfS_1 = \wh \bfP \wh \bfSigma \wh \bfQ^*$. For a given truncation tolerance $\vartheta$, set $r_1$ such that 
        \begin{align*}
            \left( \sum_{k=r_1 + 1}^{2 r} \sigma_k^2 \right)^{1/2} \leq \vartheta,
        \end{align*}
        where $\sigma_k$ are the singular values in the diagonal matrix $\wh \bfSigma$. Set $\bfP, \bfQ \in \C^{2r \times r_1}$ as the matrices of the first $r_1$ columns of $\wh \bfP$ and $\wh \bfQ$, respectively. Set $\bfU_1 = \wh \bfU \bfP$ and $\bfV_1 = \wh \bfV \bfQ$ and finally 
        $\bfS_1 = \bfP^* \wh \bfS_1  \bfQ \in \C^{r_1 \times r_1}$.
	\end{enumerate} 

          This yields an approximation of rank $r_1 \le 2 r$ in factorized form:
            $$
            \bfY_1 = \bfU_1 \bfS_1 \bfV_1^* \approx \bfA(t_1),
            $$
            where $\bfU_1\in \C^{m\times r_1}$ and $\bfV_1\in \C^{n\times r_1}$ again have orthonormal columns, and
            $\bfS_1\in \C^{r_1\times r_1}$ is invertible.
            
We emphasize that within one time step, all appearing differential equations in this algorithm can be solved in parallel, as in \cite{ceruti2023parallel}.
This version differs from the original algorithm of \cite{ceruti2023parallel} in that the original values $\widetilde \bfS_{1,\mathrm{orig}}^K = \widetilde \bfU_1^{*}\bfK(t_1)$ and $\widetilde \bfS_{1, \mathrm{orig}}^L = \bfL(t_1)^{*}\widetilde \bfV_{1}$  are replaced by first-order approximations:
\begin{align*}
    \widetilde \bfU_1^{*}\bfK(t_1) = \widetilde \bfU_1^{*}\left(\bfK_0 + h\bfF(\bfY_0)\bfV_0 + O(h^2)\right) = \widetilde \bfS_1^K + O(h^2)\,,
\end{align*}
where we used that $\widetilde \bfU_1^{*}\bfK_0 = 0$. Analogously, we obtain $\bfL(t_1)^{*}\widetilde \bfV_{1} = \widetilde \bfS_1^L + O(h^2)$. 

Note that the coefficient matrix in the $\bfS$-step is updated at rank \(r\) rather than rank \(2r\) as required by the rank-adaptive BUG integrator of \cite{ceruti2022rank}.

 As in \cite{ceruti2022rank} and \cite{ceruti2023parallel}, the modified parallel integrator has rank-adaptivity and advances all differential equations forward in time, as opposed to the projector-splitting algorithm of \cite{lubich2014projector}, in which $\bfK$-, $\bfS$- and $\bfL$-steps are done strictly sequentially in this order with a backward-in-time step for $\bfS$. 

The robust a-priori error bound obtained for the original parallel BUG integrator \cite[Theorem 4.1]{ceruti2023parallel} (see also \cite{kieri2016discretized,ceruti2022unconventional, ceruti2022rank} for related error bounds that are robust to small singular values)
extends straightforwardly to the modified parallel integrator: provided that $\Vert \bfY_0 - \bfA(t_0)\Vert_{\rm F}\le \delta$,
\begin{align*}
\Vert \bfY_k - \bfA(t_k)\Vert_{\rm F} \leq   c_1 h + c_2 \varepsilon + c_3\delta + c_4 k\vartheta,  \qquad t_0 \le t_k \le T,
\end{align*}
 where the constants \(c_i\) are independent of small singular values in both the numerical and exact solutions (but depend on $T$). Here, the vector field \(\bfF\) is assumed to be Lipschitz continuous and bounded, with a small \(\varepsilon\) model error induced by the dynamical low-rank approximation ansatz. More precisely, denoting \( \text{P}_k(Y) \)  the projection onto the tangent space \( \mathcal{T}_\bfY \mathcal{M}_{r_k} \) of the manifold of rank-\(r_k\) matrices at 
$\bfY\in \mathcal{M}_{r_k}$, and \( \text{P}_k^\perp(\bfY) = \mathrm{I}-\text{P}_k(\bfY) \) the orthogonal projection onto the normal space, the vector field is assumed to be almost tangential to the manifold for $\bfY\in \mathcal{M}_{r_k}$ in a small neighborhood of $\bfA(t)$ for $t_k\le t \le t_{k+1}$:
\[ 
    \| \mathrm{P}_k^\perp(\bfY)\bfF(\bfY) \|_{\rm F} \le \varepsilon \, .
\]
The model error \(\varepsilon\) is often unknown in practice. The family of BUG integrators allows for a computationally accessible estimation of the normal component after each time step of the algorithm, which is given by (written here for the first step $k=0$)
\begin{equation}\label{eps}
\| \mathrm{P}^\perp(\bfY_0) \bfF(\bfY_0) \| \approx \| \mathrm{P}^\perp(\bfY_0) [\wh \bfU\wh \bfU^* \bfF(\bfY_0) \wh\bfV\wh\bfV^*] \| = \| \widetilde \bfU_1^* \bfF(\bfY_0) \widetilde \bfV_1\| =: \eta \, .
\end{equation}
Computing \(\eta \) thus provides an estimate of \(\varepsilon\), which can be used to recompute a time update with a larger basis when \(\eta\) exceeds some threshold. We refer to \cite{ceruti2023parallel} for a detailed explanation of the aforementioned estimation strategy.

\section{Parallel Tucker tensor integrator}\label{sec:parallelTucker}
In the following, we present a parallel BUG integrator for problems of the form 
\begin{align*}
\dot{A}(t) = F(A(t)), \quad A(t_0) = A_0\, ,
\end{align*}
where $A(t)\in\mathbb{C}^{n_1 \times \cdots \times n_d}$ is the unknown tensor-valued solution and 
$F:\mathbb{C}^{n_1 \times \cdots \times n_d}\to\mathbb{C}^{n_1 \times \cdots \times n_d}$ is given.

The $i$-th matricization $\mat_i(A) \in \C^{n_i \times n_{\lnot i}}$ with $n_{\lnot i} = \prod_{j\neq i}n_j$ denotes the matrix where the $k$th row aligns all entries of A that have $k$ as the $i$th subscript, whereas $\ten_i$ denotes the inverse operation. 
The matricization of $F$ onto mode $i$ is denoted as 
$$
\bfF_i( \bfZ) := \textbf{Mat}_i(F(Ten_i(\bfZ))).
$$
We compute factorized low-rank tensor approximations at $t_k = t_0+kh$, 
of varying multilinear rank $\bfr^k=(r_1^k,\dots, r_d^k)$ with $r_i^k \ll n_i$ in the factorized Tucker tensor format
(see e.g. \cite{kolda2009tensor})
$$
Y^k = C^k \bigtimes_{i=1}^d \bfU_i^k \approx A(t_k),
$$
where the basis matrices $\bfU_i^k \in \C^{n_i\times r_i^k}$ have orthonormal columns and the core tensor 
$C^k\in \C^{r_1^k \times\dots \times r_d^k}$ has full multilinear rank $\bfr^k$.
We describe one time step of the method, going from $t_0$ to $t_1$. This procedure is then repeated to advance the solution approximation to $t_2,t_3,$ etc.

\subsection{Formulation of the algorithm}
Given the numerical solution in Tucker format $Y^0 = C^0 \bigtimes_{i=1}^d \bfU_i^0$ of
multilinear rank $\bfr = (r_1, \cdots, r_d)$ at time $t_0$, one step of the parallel Tucker tensor integrator from $t_0$ to $t_1$ reads:
	\begin{enumerate}
		\item Compute augmented basis matrices $\wh \bfU_i \in \C^{n_i\times 2r_i}$ $(i=1,\dots,d)$ and an updated core tensor $C(t_1)\in \C^{r_1 \times\dots \times r_d}$ (\emph{in parallel}):
		\\[2mm]
		\textbf{K$_i$-step} (for $i = 1,\dots, d$): Factorize (e.g.~by QR decomposition)
		\begin{align*}
		\textbf{Mat}_i(C^0)^\top = \bfQ_i \bfS_i^\top,
\end{align*}	
where $\bfQ_i\in \C^{r_{\neg i}\times r_i}$ with $r_{\lnot i} = \prod_{j\neq i}r_j$ has orthonormal columns
and $\bfS_i\in \C^{r_i\times r_i}$. From $t_0$ to $t_1$,
		integrate the $n_i \times r_i$ matrix differential equation
		\begin{equation}\label{K-step} 
		\dot{\textbf{K}}_i(t) = \bfF_i( \textbf{K}_i(t) \bfV_i^{0,*}) \bfV_i^0, \qquad \textbf{K}_i(t_0) = \bfU_i^0 \bfS_i^0,
		\end{equation}
		where $\bfV_i^{0,*} = \bfQ_i^{\top} \bigotimes_{j\neq i}^d \bfU_j^{0,\top} \in\C^{r_i \times n_{\lnot i}}$. 
		Construct $\wh \bfU_i = (\bfU_i^0, \widetilde \bfU_i^1) \in \C^{n_i\times 2r_i}$ as an orthonormal basis of the range 
		of the $n_i\times 2r_i$ matrix $(\bfU_i^0, \textbf{K}(t_1))$ (e.g.~by QR decomposition) such that the first $r_i$ columns equal $\bfU_i^0$. $\wh \bfU_i$ is filled with zero columns if $(\bfU_i^0, \textbf{K}(t_1))$ has rank less than $2r$. 
		Compute the tensor $\widetilde C_i^1 = hF(Y_0)\bigtimes_{j\neq i} \bfU_i^{0,*} \times_i \widetilde\bfU_i^{1,*} \in\C^{r_1\times\cdots\times r_d}$. 
		\\[2mm]
		\textbf{C-step}:
		From $t=t_0$ to $t_1$, integrate the $r_1 \times \cdots \times r_d$ tensor differential equation
		\begin{equation}\label{C-step} 
		\dot{C}(t) = F\Bigl( C(t) \bigtimes_{\ell=1}^d\bfU_{\ell}^{0}\Bigr) \bigtimes_{j=1}^d\bfU_j^{0, *}, 
		\qquad C(t_0) = C^0.
		\end{equation}
		\item \textbf{Augment}:
		Construct the augmented core tensor $\widehat C^1 \in\C^{2r_1\times\cdots\times 2r_d}$ such that
        \begin{align}
    		&\widehat C^1 \bigtimes_{j=1}^d (\bfI_{r_j}, \bm{0}_{r_j}) = {C}(t_1)\,,\label{eq:barC}\\
    		&\widehat C^1 \bigtimes_{j\neq i} (\bfI_{r_j}, \bm{0}_{r_j}) \times_i ( \bm{0}_{r_i}, \bfI_{r_i}) = \widetilde C_i^1 \quad \text{for } i = 1,\dots,d\, ,\label{eq:tildeC}
		\end{align} 
		and zero entries elsewhere. Note $(\bfI_{r_j}, \bm{0}_{r_j})= \bfU_j^{0, *}\wh \bfU_j$ and 
  $( \bm{0}_{r_i}, \bfI_{r_i})= \widetilde \bfU_i^{1, *}\wh \bfU_i$.
        \item \textbf{Truncate:} For $i=1,\dots,m$ compute (in parallel) the reduced singular value decompositions $\mat_i(\wh C^1) = \wh \bfP_i \wh \bfSigma_i \wh \bfQ_i^*$. Then set $r_i^1$ such that 
        \begin{align*}
            \left( \sum_{k=r_i^1 + 1}^{2 r_i} \sigma_k^2 \right)^{1/2} \leq \vartheta,
        \end{align*}
        where $\sigma_k$ are the singular values in the diagonal matrix $\wh \bfSigma_i$. Set $\bfP_i \in \C^{2r_i \times r_i^1}$ as the matrix of the first $r_i^1$ columns of $\wh \bfP_i$. Set $\bfU_i^1 = \wh \bfU_i \bfP_i$ and finally $C^1 = \wh C^1 \bigtimes_{i=1}^m \bfP_i^* \in \C^{r_1^1 \times \dots \times r_d^1}$.
	\end{enumerate} 
	Like in the matrix version of Section~\ref{sec:recapParallel}, all differential equations can be solved fully in parallel. Further, we note that the filling with zero columns in the $\wh \bfU_i$ matrices is not necessary to implement the algorithm, but makes the presentation simpler.
 
 We remark that this integrator is similar to the rank-adaptive BUG integrator for Tucker tensors as proposed in \cite{ceruti2022rank}. The central difference between the two integrators is that the parallel integrator performs a rank-$\mathbf{r}$ coefficient update in parallel, followed by the augmentation step, whereas the rank-adaptive BUG integrator instead performs a sequential rank-$2\mathbf{r}$ core tensor update.
 Before the truncation step, the rank-adaptive and the parallel BUG integrators only differ in the updated core tensor. That is, denoting the solution of the rank-adaptive BUG integrator as $\widehat Y_{\mathrm{BUG}}^1$ we have $\widehat Y_{\mathrm{BUG}}^1 = \widehat C^1_{\mathrm{BUG}} \bigtimes_{i=1}^d \wh\bfU_i$ compared to the solution of the parallel BUG integrator $\widehat Y^1 = \widehat C^1 \bigtimes_{i=1}^d \wh\bfU_i$. In the following, we use this observation to derive a robust error bound by bounding the distance between these two integrators.
	\subsection{Robust error bound}
	The presented parallel Tucker integrator inherits the robust error bound of the rank-adaptive BUG low-rank matrix integrator \cite{ceruti2022rank} for Tucker tensors. Let $\mathcal{V} = \C^{n_1 \times \dots \times n_d}$ and $\mathcal{M}^k$ the manifold of tensors of multilinear rank $\bfr^k$ for the $k$th time step. We assume that
 with respect to the tensor norm $\|\cdot\|$, which is the Euclidean norm of the vector of tensor entries,
	\begin{enumerate}
	\item $F: \mathcal{V} \rightarrow \mathcal{V}$ is Lipschitz-continuous and bounded, i.e.,
    \begin{align*}
        \norm{F(Y) - F(\wt Y)} &\leq L \norm{Y - \wt Y} \quad &&\text{for all} \ Y, \wt Y \in \mathcal{V}\\
        \norm{F(Y)} &\leq B \quad &&\text{for all} \ Y \in \mathcal{V}.
    \end{align*}
    \item For $Y\in\mathcal{M}^k$ near the exact solution $A(t)$ for $t\in [t_k,t_{k+1}]$ we assume, with the
    orthogonal projection $P^k(Y)$ onto the tangent space $\mathcal{T}_Y \mathcal{M}^k$ and the normal projection $P^k(Y)^\perp=I-P^k(Y)$,
    \begin{align*}
        \norm{P^k(Y)^\perp F(Y) } \leq \eps.
    \end{align*}
	\item The error of the initial value is bounded by $\Vert Y_0 - A(0)\Vert \leq \delta$.
	\end{enumerate}
	Then, there is the following robust error bound.
	\begin{theorem}\label{thm:error_bound_tucker}
		Under assumptions 1. to 3., the error of the parallel Tucker integrator is bounded by
		\begin{align*}
		\Vert Y_k - A(t_k) \Vert  \le c_1 h + c_2 \varepsilon + c_3 \delta + c_4 k\vartheta,
  \qquad t_0 \le t_k \le T,
		\end{align*}
		where all $c_i$ only depend on the bound and Lipschitz constant of $F$ and on $T$.
  In particular, the $c_i$ are independent of the singular values of matricizations of the core tensor.
	\end{theorem}
	\begin{proof}
		We notice that the local error can be bounded by
		\begin{align*}
		\Vert Y^1 - A(t_1) \Vert \le 
  \Vert Y^1 - \wh Y^1 \Vert + \Vert \wh Y^1 - A(t_1) \Vert \le \vartheta +
   \Vert\wh  Y^1 - \widehat Y_{\mathrm{BUG}}^1 \Vert + \Vert \wh Y_{\mathrm{BUG}}^1 - A(t_1) \Vert  .
		\end{align*}
		By \cite{ceruti2022rank}, the local error of the rank-adaptive BUG integrator is bounded by
  $$
  \Vert \wh Y_{\mathrm{BUG}}^1 - A(t_1) \Vert \leq c_1h^2 + c_2h\varepsilon+c_3 h\delta,
  $$
  and so we only need to bound
		\begin{align*}
			\Vert \wh Y^1 - \widehat Y_{\mathrm{BUG}}^1 \Vert = \Vert \widehat C^1 - \widehat C^1_{\mathrm{BUG}} \Vert\, .
		\end{align*}
		We now investigate the norms of individual blocks in $\widehat C^1 - \widehat C^1_{\mathrm{BUG}}$. Here, we need to investigate three parts: The block that belongs to \eqref{eq:barC}, blocks belonging to \eqref{eq:tildeC}, and blocks that are set to zero.
  \begin{enumerate}
      \item To investigate the local error of the integrator in the block that belongs to \eqref{eq:barC}, we define $\bar Y(t) :=  C(t)\bigtimes_{i=1}^d \bfU_i^0$, $\wh Y_{\mathrm{BUG}}(t) :=\wh C_{\mathrm{BUG}}(t) \bigtimes_{i=1}^d\wh\bfU_i$, and $\bar F(t) := F( \bar Y(t))$ and
      $\wh F(t) := F( \wh Y_{\mathrm{BUG}}(t))$. Then, we have 
    \begin{align*}
        &\Vert (\widehat C^1_{\mathrm{BUG}} - \widehat C^1) \bigtimes_{j=1}^d \bfU_j^{0,*}\widehat \bfU_j\Vert 
        \\
        &= \| \int_{t_0}^{t_1} \Bigl( \dot{\widehat C}_{\mathrm{BUG}}(t) \bigtimes_{j=1}^d \bfU_j^{0,*}\widehat \bfU_j -
        \dot C (t) \Bigr)dt \|
        \\
        &= \| \int_{t_0}^{t_1} \Bigl(  \wh F(t) \bigtimes_{j=1}^d \wh\bfU_j^{*} \bigtimes_{j=1}^d \bfU_j^{0,*}\widehat \bfU_j -
        \bar F(t)\bigtimes_{j=1}^d \bfU_j^{0,*} \Bigr)dt \|
        \\
        &\le  \int_{t_0}^{t_1}\Vert \wh F(t) \bigtimes_{j=1}^d \bfU_j^{0,*}\widehat \bfU_j\wh\bfU_j^{*}  - \bar F(t)\bigtimes_{j=1}^d \bfU_j^{0,*}\Vert\,dt
        \\
        &= \int_{t_0}^{t_1}\Vert (\wh F(t) - \bar F(t)) \bigtimes_{j=1}^d \bfU_j^{0,*}\Vert\,dt
        \\
        &\le L\int_{t_0}^{t_1}\Vert \wh Y_{\mathrm{BUG}}(t) - \bar Y(t)\Vert\,dt \leq 2LBh^2,
    \end{align*}
    since $\|Y_{\mathrm{BUG}}(t)-Y_0\|\le Bh$ and $\|\bar Y(t)-Y_0\|\le Bh$ for $t_0\le t \le t_1$ by the assumed bound of~$F$.
    \item Second, we bound the local error of the integrator in the blocks that belong to \eqref{eq:tildeC}:
    \begin{align*}
    	&\Vert (\widehat C^1_{\mathrm{BUG}} - \widehat C^1) \bigtimes_{j\neq i} \bfU_j^{0,*}\widehat \bfU_j\times_i \widetilde\bfU_i^{1,*}\widehat \bfU_i\Vert 
     \\
     &\leq \int_{t_0}^{t_1}\Vert (\wh F(t) - F(Y_0)) \bigtimes_{j\neq i} \bfU_j^{0,*} \times_i \widetilde\bfU_i^{1,*}\Vert\,dt
     \\
        &\leq L\int_{t_0}^{t_1}\Vert \wh Y_{\mathrm{BUG}}(t) - Y_0\Vert\,dt \leq LBh^2\,.
    \end{align*}
    \item All remaining terms in $\wh C^1$ are zero-valued. Hence, for any sets of indices $\mathcal{I}\subset \{1,\cdots,d\}$ with $|\mathcal{I}| \geq 2$ and $\mathcal{J} := \{1,\cdots,d\}\backslash \mathcal{I}$, we wish to bound the terms
        \begin{align*}
    	\ostar := \Vert \widehat C^1_{\mathrm{BUG}} \bigtimes_{j\in \mathcal{J}} \bfU_j^{0,*}\widehat \bfU_j\bigtimes_{i\in \mathcal{I}} \widetilde\bfU_i^{1,*}\widehat \bfU_i\Vert \leq\,& \int_{t_0}^{t_1}\Vert \wh F(t)\bigtimes_{j\in \mathcal{J}} \bfU_j^{0,*}\bigtimes_{i\in \mathcal{I}} \widetilde\bfU_i^{1,*}\Vert\,dt\,.
    \end{align*}
    Since the orthogonal projection onto the tangent space at $Y = C\bigtimes_{j=1}^d \bfU_j$ equals (see \cite{koch2010dynamical})
    $$
    P(Y)Z := Z \bigtimes_{j=1}^d \bfU_j\bfU_j^{*} + \sum_{i = 1}^d Z \bigtimes_{j\neq i} \bfU_j\bfU_j^{*} \times_i (\bfI - \bfU_i\bfU_i^{*})\,,
    $$
    we find that $P(Y^0)\wh F(t_0)\bigtimes_{i\in \mathcal{I}} \widetilde\bfU_i^{1,*} = 0$ and therefore 
    \begin{align*}
    	\ostar \leq\,& \int_{t_0}^{t_1}\Vert (\wh F(t) - P(Y^0)\wh F(t_0) )\bigtimes_{j\in \mathcal{J}} \bfU_j^{0,*}\bigtimes_{i\in \mathcal{I}} \widetilde\bfU_i^{1,*}\Vert\,dt \\
    	\leq\,& \int_{t_0}^{t_1}\Vert (\wh F(t) - \wh F(t_0) )\bigtimes_{j\in \mathcal{J}} \bfU_j^{0,*}\bigtimes_{i\in \mathcal{I}} \widetilde\bfU_i^{1,*}\Vert\,dt + h \varepsilon \leq ch^2 + h \varepsilon\,.
    \end{align*}
  \end{enumerate}
This yields 
$$
\Vert \wh Y^1 - \widehat Y_{\mathrm{BUG}}^1 \Vert \le \wh c h^2 + h\eps,
$$
and hence the local error is bounded by 
$$
\Vert Y^1 - A(t_1) \Vert \leq c_1 h^2 + c_2 h \varepsilon + c_3 h\delta + \vartheta.
$$
We then pass to the global error via Lady Windermere's fan \cite[Sections  I.7 and II.3]{HNW91}, which yields the stated error bound (with different constants $c_i$).
\end{proof}

\section{Parallel Tree Tensor Network integrator}\label{sec:TTN}
\subsection{Tree tensor networks}
In this subsection, we want to recap the tree tensor network (TTN) formalism from \cite{CLW2021,CLS2023}. TTNs are a hierarchical, data-sparse tensor format. Trees encode the hierarchical structure and are defined as follows.
\begin{definition}[{{\cite[Definition 2.1]{CLW2021} Ordered trees with unequal leaves}}]
	Let $\mathcal{L}$ be a given finite set, the elements of which are referred to as leaves.
	We  define the set $\mathcal{T}$ of trees  $\tau$ with the corresponding set of leaves $L(\tau)\subseteq \mathcal{L}$ recursively as follows:
	\begin{enumerate}
		\item[(i)] {\em Leaves are trees:} 
		$ \mathcal{L} \subset \mathcal{T}$,\ \text{ and }\  $L(l) := \{l\}$ for each $l \in \mathcal{L}$.
		\item[(ii)]  {\em Ordered $m$-tuples of trees with different leaves are again trees:} 
		If, for some $m\ge 2$,
		$$
		\tau_1, \dots, \tau_m \in \mathcal{T} 
		\quad \text{ with }\quad
		L(\tau_i ) \cap L(\tau_j ) = \emptyset \quad \forall i \neq j,
		$$
		 then their ordered $m$-tuple is in $\mathcal{T}$:		
		$$ \tau := (\tau_1, \dots, \tau_m) \in \mathcal{T} 
		, \quad \text{ and }\quad 
		L(\tau) := \dot{\bigcup}_{i=1}^m L(\tau_i) \ . 
		$$
	\end{enumerate}		
\end{definition}
The trees $\tau_1,\dots,\tau_m$ are called direct subtrees of a tree $\tau = (\tau_1, \dots, \tau_m)$. 
Together with direct subtrees of direct subtrees of $\tau$ etc. they are called the subtrees of $\tau$.
This admits a partial ordering for two trees $\sigma,\tau$ by setting
\begin{align*}
    & \sigma \le \tau \ \text{if and only if} \ \sigma \ \text{is a subtree of} \ \tau. \\
    &\sigma < \tau \ \text{if and only if} \ \sigma \ \text{is a subtree of} \ \tau \ \text{and} \ \sigma \neq \tau.
\end{align*}
Now fix a tree $\bar \tau$. With the $l$th leaf of $\bar \tau$ we associate a basis matrix $\bfU_l \in \C^{n_l \times r_l}$ of full rank, while with each tree $\tau = (\tau_1, \dots, \tau_m)$ we associate a connecting tensor $C_\tau$ of full multilinear rank $r= (r_\tau,r_{\tau_1},\dots,r_{\tau_m})$. We set $r_{\bar \tau} = 1$ at the root tensor. Tree tensor networks are now defined recursively from the root to the leaves.
\begin{definition}[{{\cite[Definition 2.2]{CLW2021} Tree tensor network}}]
	For a given tree $\bar \tau \in \mathcal{T}$ and basis matrices $\bfU_l$ and connecting tensors $C_\tau$ as described above, we recursively define a tensor $X_{\bar \tau}$ with a tree tensor network representation (or briefly a tree tensor network) as follows: 
	\begin{enumerate}
		\item[(i)]
		For each leaf  $\, l \in \mathcal{L}$, we set 
		$$X_l := \bfU_l^\top \in \C^{r_l \times n_l} \ .  $$
		\item[(ii)]
		For each subtree $\tau = ( \tau_1, \dots, \tau_m)$  (for some $m\ge 2$) of $\bar\tau $, 
		we set \\
		$n_\tau = \prod_{i=1}^m  n_{\tau_i}$ and $\bfI_\tau$ the identity matrix of dimension $r_\tau$, and
		\begin{align*}
		& X_\tau := C_\tau \times_0 \bfI_{\tau} \bigtimes_{i=1}^m \bfU_{\tau_i} 
		\in \mathbb{C}^{r_\tau \times n_{\tau_1} \times \dots \times n_{\tau_m}},
		\\
		& \bfU_{\tau} := \mat_0( X_{\tau} )^\top \in \mathbb{C}^{n_\tau \times r_\tau} \ .
		\end{align*}
		The subscript $0$ in $\times_0$ and $\mat_0( X_{\tau} )$ refers to the mode $0$ of dimension $r_{\tau}$ in $\C^{r_{\tau} \times r_{\tau_1}\times\dots\times r_{\tau_m}} $.
	\end{enumerate}
	The tree tensor network $X_{\bar{\tau}}$ (more precisely, its representation in terms of the matrices $\bfU_\tau$) is called {\em orthonormal} if for each subtree $\tau < \bar{\tau}$, the matrix $\bfU_\tau$ has orthonormal columns. 
\end{definition}
Binary tree tensor networks are studied in the mathematical literature as hierarchical Tucker tensors \cite{hackbusch2012tensor} and for general trees in \cite{Falco2018}. In the physical literature, the special class of matrix product states/tensor trains are widely used for applications in quantum physics \cite{Vidal2003,haegeman2016unifying,Schollwoeck2011}. In the chemical literature, tree tensor networks are used in the framework of the multilayer multiconfiguration time-dependent Hartree method (ML-MCTDH) \cite{Wang2003}. \\
\paragraph*{Constructing reduced operators and initial data for subtrees}
Suppose we have a function $F_{\bar \tau}$ which maps tree tensor networks to tree tensor networks and an initial TTN $Y_{\bar \tau}^0$. In the formulation of the algorithm, we will need a reduced version of the function and the initial data for each subtree. We will briefly explain how to construct $F_\tau$ and $Y_\tau^0$ for arbitrary subtrees $ \tau \leq \bar\tau$. For that we follow the construction from \cite{CLW2021}, where also a more detailed description can be found. The construction works recursively from the root to the leaf. \\
For a tree $\tau=(\tau_1,\dots,\tau_m) \leq \bar \tau$ we define the tensor space $\mathcal{V}_\tau=\C^{r_\tau \times n_{\tau_1} \times \dots \times n_{\tau_m}}$ and the manifold $\mathcal{M}_\tau = \mathcal{M}(\tau,(n_l)_{l\in L(\tau)},(r_\sigma)_{\sigma \leq \tau}) \subset \mathcal{V}_\tau$.  
Suppose that the function $F_\tau:\mathcal{V}_\tau \rightarrow \mathcal{V}_\tau$ and the initial data $Y_\tau^0\in\mathcal{M}_\tau$ are already constructed, and is of the form
\begin{align*}
    Y_\tau^0 = C_\tau^0 \times_0 \bfI_\tau \bigtimes_{i=1}^m \bfU_{\tau_i}^0.
\end{align*}
We now construct the reduction of $F_\tau$ and $Y_\tau^0$ to the subtrees $\tau_i$ ($i=1,\dots,m$).
We consider the QR decomposition $\mat_i(C_\tau)^\top = \bfQ_{\tau_i}\bfR_{\tau_i}$ and the matrix 
\begin{align*}
    \bfV_{\tau_i}^{0,*} := \mat_i \left( \ten_i(\bfQ_{\tau_i}^\top) \times_0 \bfI_\tau \bigtimes_{j \neq i} \bfU_{\tau_j}^0 \right).
\end{align*}
Following \cite{CLW2021}, we define the prolongation $\pi_{\tau,i}$ and restriction $\pi_{\tau,i}^{\dagger}$
\begin{align*}
    \pi_{\tau,i}(Y_{\tau_i}) &:= \ten_i ( \mat_0 (Y_{\tau_i})^\top \bfV_{\tau_i}^{0,*}) \in \mathcal{V}_\tau \quad \text{for } Y_{\tau_i} \in \mathcal{V}_{\tau_i} \\
	\pi_{\tau,i}^{\dagger} (Z_\tau) &:= \ten_0(\bfV_{\tau_i}^{0,\top} \mat_i(Z_\tau)^\top) \in \mathcal{V}_{\tau_i} \quad \text{for } Z_{\tau} \in \mathcal{V}_\tau.
\end{align*}
Altogether, we obtain the reduced function and initial data by 
\begin{align*}
    F_{\tau_i} &:= \pi_{\tau,i}^{\dagger} \circ F_\tau \circ \pi_{\tau,i} \\
    Y_{\tau_i}^0 &:= \pi_{\tau,i}^{\dagger}(Y_\tau^0).
\end{align*}
The computation of the prolongation and restriction can be efficiently done by a recursion, see \cite[Sec. 4.4]{CLW2021} for details.

\subsection{Parallel BUG integrator for tree tensor networks}
Consider a tree $\bar \tau = (\bar \tau_1,\dots,\bar \tau_m)$ and a corresponding tree tensor network
\begin{align*}
    Y_{\bar \tau}^0 = C_{\bar \tau}^0 \times_0 \bfI_{\bar \tau} \bigtimes_{i=1}^m \bfU_{\bar \tau_i}^0
\end{align*}
at time $t_0$. To evolve $Y_{\bar \tau}^0$ in time, one has to update all basis matrices $\bfU_l^0$, with $ l \in \mathcal{L}$, and all connecting tensors $C_\tau^0$, with $\tau \leq \bar \tau$. \emph{The algorithm presented here (Algorithm \ref{alg:Parallel_TTN_BUG}) evolves all basis matrices and connecting tensors fully in parallel.} A basis matrix is updated by the subflow $\Phi_l$ (Algorithm \ref{alg:Phi-tau-i}), a connecting tensor by the subflow $\Psi_\tau$ (Algorithm \ref{alg:Psi-tau}). The evolution of all basis matrices and connecting tensors is followed by an augmentation strategy (Algorithm \ref{alg:Augmentation}), which is the natural extension of the augmentation step for matrices \cite[sec. 3.1]{ceruti2023parallel} and Tucker tensors from above to tree tensor networks. Note that the augmentation process is a recursive procedure going from the leaves to the root. The augmentation does not involve any differential equations. The ranks after augmentation are (usually) doubled in each time step, i.e., $\wh r_\tau^1 = 2r_\tau^0$. Therefore, a rank truncation algorithm is applied after each time step; see \cite[Algorithm 7]{CLS2023} for details.

\bigskip
\begin{algorithm2e}[H]
	\caption{Parallel TTN BUG}
	\label{alg:Parallel_TTN_BUG}
	\SetAlgoLined
	\KwData{tree $\bar \tau = ({\bar \tau}_1,\dots ,{\bar \tau}_m)$, TTN $Y_{\bar \tau}^0 = C_{\bar \tau}^0 \times_0 \bfI_{\bar \tau} \bigtimes_{i=1}^m \bfU_{{\bar \tau}_i}^0$ in factorized form with tree ranks $(r_\tau^0)_{\tau \leq {\bar \tau}}$, functions $(F_{\tau})_{\tau \leq \bar \tau}$, $t_0,t_1$, truncation tolerance $\vartheta$}
	\KwResult{TTN $ Y_{\bar \tau}^1 = C_{\bar \tau}^1 \times_0 \bfI_{\bar \tau} \bigtimes_{i=1}^m \bfU_{{\bar \tau}_i}^1$ in factorized form with tree ranks $(r_\tau^1)_{\tau \leq \bar \tau}$}
	\Begin{
		 \For{$\tau \leq \bar \tau$ in parallel }{
            \uIf{$\tau =l$ is a leaf}{
                Set $\sigma = (\sigma_1,\dots,\sigma_m)$ such that $\sigma_i = l$ for an $i \in \{1,\dots,m\}$ 
                
                $\%$ \textit{Find the parent node in the tree.} 
                
                Set $\wh \bfU_\tau = \Phi_l(\sigma,l,C_\sigma^0,\bfU_l^0,F_l,t_0,t_1)$ 
                
                $\%$ \textit{Update the basis matrices, see Algorithm \ref{alg:Phi-tau-i}}
            }
            \Else{
                Set $\bar C_\tau^1 = \Psi_\tau (\tau,Y_\tau^0,F_\tau,t_0,t_1)$ 
                
                $\%$ \textit{Update the connecting tensors, see Algorithm \ref{alg:Psi-tau}}
            }
		}
        Set $\wh Y_{\bar\tau}^1 = \mathcal{A}_{\bar \tau} (\bar \tau,(\bar C_\tau^1)_{\tau \leq \bar \tau}, Y_{\bar\tau}^0, (\wh \bfU_l)_{l \in \mathcal{L}},(F_{\tau})_{\tau \leq \bar \tau},h)$ 
        
        $\%$ \textit{Augment the updated TTN, see Algorithm \ref{alg:Augmentation}} 
        
        Set $Y_{\bar\tau}^1 = \Theta (\wh Y_{\bar\tau}^1,\vartheta)$ 
        
        $\%$ \textit{Truncation with tolerance $\vartheta$, see \cite[Algorithm 7]{CLS2023}}
  }
\end{algorithm2e}

\begin{algorithm2e}[H]
	\caption{Subflow $\Phi_{l}$ (update a basis matrix)}
	\label{alg:Phi-tau-i}
	\SetAlgoLined
	\KwData{tree $\tau = (\tau_1,\dots ,\tau_m)$ with $\tau_i = l$, $C_\tau^0$ connecting tensor directly connected to $\bfU_l^0$, function $F_{l}(\cdot)$, $t_0,t_1$
	}
	\KwResult{ $\widehat{\bfU}_l = [\bfU_l^0,\widetilde{\bfU}_l^1 ] \in \C^{n_l \times 2r_l^0} $	}
	\Begin{
		compute a QR-decomposition $\mat_i({C_{\tau}^0})^\top = \bfQ_{l}^0\bfS_{l}^{0,\top}$; 
  
		set $\bfY_{l}^0 = \bfU_{l}^{0,\top} \times_0 \bfS_{l}^{0,\top} $ \\
		
			solve the $n_l \times r_l^0$ matrix differential equation
			\begin{align*}
				\dot{\bfY}_{l}(t) = F_{l}(\bfY_{l}(t)), \ \ \ \bfY_{l}(t_0) = \bfY_{l}^0\in \C^{r_l^0\times n_l}\,;
			\end{align*}
   
			compute $\wh{\bfU}_{l} \in \C^{n_l \times 2r_l^0}$ as an orthonormal basis of the range of the $n_l \times 2{r}_l^0$ matrix  $\bigl( \bfU_{l}^{0},\bfY_{l}(t_1)^\top \bigr) $ such  that the first $r_l^0$ columns of $\wh{\bfU}_{l}$ equal $\bfU_l^0$. $\wh \bfU_l$ is filled with zero columns if $\bigl( \bfU_{l}^{0},\bfY_{l}(t_1)^\top \bigr)$ has rank less than $2r$.  \\
		}
\end{algorithm2e}

\begin{algorithm2e}[H]
	\caption{Subflow $\Psi_\tau$ (update the connecting tensor)}
    \label{alg:Psi-tau}
	\SetAlgoLined
	\KwData{tree $\tau = (\tau_1,\dots ,\tau_m)$, connecting tensor $C_\tau^0\in \C^{r_\tau^0 \times r_{\tau_1}^0\times\dots\times r_{\tau_m}^0}$, basis matrices $\bfU_{\tau_i}^0$ in factorized form such that $Y_\tau^0 = C_\tau^0 \times_0 \bfI \bigtimes_{i=1}^m \bfU_{\tau_i}^0$,  function $F_\tau(\cdot)$, $t_0,t_1$}
	\KwResult{connecting tensor $ \bar{C}_\tau^1 \in \C^{r_\tau \times r_{\tau_1} \times \dots  \times r_{\tau_m}}$}
		\Begin{
		solve the $ r_\tau \times r_{\tau_1} \times \dots  \times r_{\tau_m}$ tensor differential equation from $t_0$ to $t_1$
		\begin{equation*}
			\dot{\bar{C}}_\tau(t) = F_\tau \bigl(\bar{C}_\tau(t) \bigtimes_{i=1}^{m} \bfU_{\tau_i}^0\bigr) \bigtimes_{i=1}^{m} \bfU_{\tau_i}^{0,*}, \quad \bar C_\tau(t_0) =  C_\tau^0\,;	
		\end{equation*}

	set  $\bar{C}_\tau^1 = \bar{C}_\tau(t_1)$	
	}
\end{algorithm2e}

\subsection{Augmentation step}
The augmentation step is given as\\
\begin{algorithm2e}[H]
	\caption{Augmentation $\mathcal{A}_\tau$}
    \label{alg:Augmentation}
	\SetAlgoLined
	\KwData{tree $\tau = (\tau_1,\dots ,\tau_m)$, connecting tensor $\bar{C}_\tau^1 \in \C^{r_\tau^0 \times r_{\tau_1}^0\times\dots\times r_{\tau_m}^0}$, connecting tensors $(\bar C_\sigma^1)_{\sigma < \tau}$, tree tensor network $Y_\tau^0$, basis matrices $\wh{\bfU}_{l}$ for $l \in \mathcal{L}$, functions $(F_{\sigma}(\cdot))_{\sigma \leq  \tau}$, time step size $h$}
	\KwResult{augmented TTN $\wh{Y}_\tau^1 = \wh C_\tau^1 \times_0 \bfI_{\tau} \bigtimes_{i=1}^m \wh \bfU_{\tau_i}$ in factorized form of tree rank $( 2 r_\sigma^0)_{\sigma \leq \tau}$ }
		\Begin{
        $\%$ Augmentation of subtrees \\
        \For{$i=1:m$ }{
            \uIf{$\tau_i \notin \mathcal{L}$, i.e. $\tau_i = (\sigma_1,\dots,\sigma_k)$}{
                $\wh Y_{\tau_i}^1$ = \textit{Augmentation}$(\tau_i,(\bar C_{\varsigma}^1)_{\varsigma \leq \tau_i}, Y_{\tau_i}^0,(\wh \bfU_{l})_{l \in \mathcal{L}},(F_{\varsigma}(\cdot))_{\varsigma \leq  \tau_i},h)$ 
                
                Set $\wh{C}_{\tau_i}^0 = C_{\tau_i}^0 \bigtimes_{j=1}^k \wh{\bfU}_{\sigma_j}^{*} \bfU_{\sigma_j}^0 $ 
                
                Compute an orthonormal basis $\wh{\bfQ}_{\tau_i}$ of the range of $(\mat_0(\wh{C}_{\tau_i}^0)^\top,\mat_0(\wh{C}_{\tau_i}^1)^\top)$ such that the first $r_{\tau_i}^0$ columns of $\wh{\bfQ}_{\tau_i}$ equal $\mat_0(\wh{C}_{\tau_i}^0)^\top$. $\wh{\bfQ}_{\tau_i}$ is filled
                with zero columns if  the matrix $(\mat_0(\wh{C}_{\tau_i}^0)^\top,\mat_0(\wh{C}_{\tau_i}^1)^\top)$ has rank less than $2r_{\tau_i}^0$. 
                
                Set $\wh{\bfU}_{\tau_i} = \mat_0(\wh{X}_{\tau_i}^1)^\top$, where the orthonormal TTN $\wh{X}_{\tau_i}^1$ is obtained from $\wh{Y}_{\tau_i}^1$ by replacing the connecting tensor with $\wh{C}_{\tau_i}^1 = \ten_0(\wh{\bfQ}_{\tau_i}^\top)$; 
                
                $\%$ Construction of $\wt \bfU_{\tau_i}^1$ 
                
                
                Set $\bfM_{\tau_i}$ as the last $ r_{\tau_i}^0 $ columns of $\wh{\bfQ}_{\tau_i}$;
                
                Set $\wt \bfU_{\tau_i}^1 = \mat_0 (\wt X_{\tau_i}^1)^\top$, where the TTN $\wt X_{\tau_i}^1$ consists of the connecting tensor $\ten_0(\bfM_{\tau_i}^\top)$ and $\wt \bfU_{\sigma_j}^1 = \wh \bfU_{\sigma_j}$, for $j=1,\dots,k$; 
            }
            Compute the tensor $\widetilde C_{\tau_i}^1 = hF(Y_\tau^0)\bigtimes_{j\neq i} \bfU_{\tau_j}^{0,*} \times_i \widetilde\bfU_{\tau_i}^{1,*}$;
        }
        
        $\%$ Augmentation of connecting tensor
        
        Construct the augmented connecting tensor $\widehat C_\tau^1 \in\C^{r_\tau \times 2r_{\tau_1}\times\cdots\times 2r_{\tau_m}}$ such that
        \begin{align}
    		&\widehat C_\tau^1 \bigtimes_{i=1}^m (\bfI_{r_{\tau_i}}, \bm{0}_{r_{\tau_i}}) = \bar{C}_\tau^1\, \label{eq:barC_TTN_case}\\
    		&\widehat C_\tau^1 \bigtimes_{j\neq i} (\bfI_{r_{\tau_j}}, \bm{0}_{r_{\tau_j}}) \times_i ( \bm{0}_{r_{\tau_i}}, \bfI_{r_{\tau_i}}) = \widetilde C_{\tau_i}^1 \quad \text{for } i = 1,\dots,m\, .\label{eq:tildeC_TTN_case}
		\end{align} 
	}
\end{algorithm2e}
\bigskip
We try to give more intuition about the augmentation step. In the subflow $\mathcal{A}_\tau$, each connecting tensor $\bar C_\tau^1$ is sequentially augmented in the $i$th mode by a block $\wt C_{\tau_i}^1$. This procedure is graphically illustrated for a connecting tensor of a binary tree, i.e. an order three tensor, in the left part of figure \ref{fig:connecting_tensor_augment}. The augmentation in the $0$-dimension happens after the recursion in the augmentation of the subtrees. The output of the recursion at the top node is a connecting tensor, which is augmented in all dimensions except the $0$-dimension. It is then augmented in the $0$-dimension, see the right part of figure \ref{fig:connecting_tensor_augment}. All remaining blocks are set to zero, which is why the parallel TTN BUG gives a rougher approximation than the rank-adaptive TTN BUG, where all entries of the augmented connecting tensor are (usually) non-zero. See the numerical example section for a more detailed comparison.\\

Analogous to the Tucker case, it holds  $(\bfI_{r_{\tau_i}^0}, \bm{0}_{r_{\tau_i}^0})= \bfU_{\tau_i}^{0, *}\wh \bfU_{\tau_i}$ and 
$( \bm{0}_{r_{\tau_i}^0}, \bfI_{r_{\tau_i}^0})= \widetilde \bfU_{\tau_i}^{1, *}\wh \bfU_{\tau_i}$. Further, filling up with zero columns is not necessary for an actual implementation, but makes the presentation simpler. We remark that without the zero columns, the matrices $\wt \bfU_{\tau_i}^1$ can be empty. Then the connecting tensor is not augmented in the corresponding mode. The augmentation of the connecting tensor can also be expressed in a sequential way by matricizations and tensorizations. For a $\tau = (\tau_1,\dots,\tau_m)$, suppose we already have computed all $\wt \bfU_{\tau_i}^1$ and $\widetilde{C}_{\tau_i}^1$, $i=1,\dots,m$. Then the augmentation can be implemented by

\begin{algorithm2e}[H]
    \SetAlgoLined
		\Begin{
        $\%$ Augmentation of connecting tensor
  
        Set $\widehat{C}_\tau^1 = \bar{C}_\tau^1$ 
        
		\For{$i=1:m$ {\rm }}{
                
                Set $\widehat{C}_\tau^1 = \ten_i \begin{pmatrix}
                \mat_i(\widehat{C}_\tau^1) \\
                \mat_i(\widetilde{C}_{\tau_i}^1)
               \end{pmatrix} $ 
            }
        }
\end{algorithm2e}

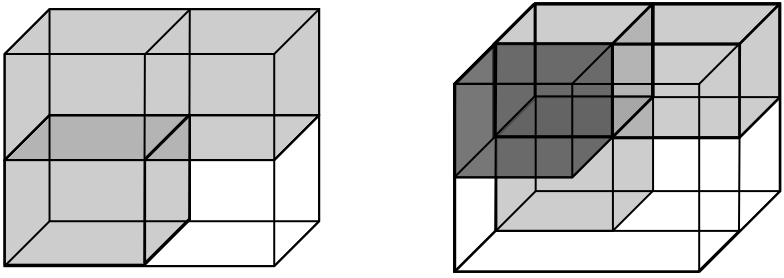
\begin{figure}
    \centering
    \tikzset{every picture/.style={line width=0.75pt}} 
    \begin{minipage}{0.45\textwidth}
    
    \begin{tikzpicture}[x=0.75pt,y=0.75pt,yscale=-1,xscale=1,scale=0.55]

\draw  [fill={rgb, 255:red, 155; green, 155; blue, 155 }  ,fill opacity=0.5 ] (238.38,53.62) -- (279.57,12.44) -- (526.8,12.44) -- (526.8,109.77) -- (485.62,150.95) -- (238.38,150.95) -- cycle ; \draw   (526.8,12.44) -- (485.62,53.62) -- (238.38,53.62) ; \draw   (485.62,53.62) -- (485.62,150.95) ;
\draw  [fill={rgb, 255:red, 155; green, 155; blue, 155 }  ,fill opacity=0.5 ] (238.38,150.92) -- (279.54,109.77) -- (408.38,109.77) -- (408.38,205.79) -- (367.23,246.95) -- (238.38,246.95) -- cycle ; \draw   (408.38,109.77) -- (367.23,150.92) -- (238.38,150.92) ; \draw   (367.23,150.92) -- (367.23,246.95) ;
\draw   (526.8,109.77) -- (485.62,150.95) -- (238.38,150.95) -- (238.38,53.62) -- (279.57,12.44) -- (526.8,12.44) -- cycle ; \draw   (238.38,150.95) -- (279.57,109.77) -- (526.8,109.77) ; \draw   (279.57,109.77) -- (279.57,12.44) ;
\draw   (526.8,207.09) -- (485.62,248.28) -- (238.38,248.28) -- (238.38,150.95) -- (279.57,109.77) -- (526.8,109.77) -- cycle ; \draw   (238.38,248.28) -- (279.57,207.09) -- (526.8,207.09) ; \draw   (279.57,207.09) -- (279.57,109.77) ;
\draw   (238.38,150.95) -- (279.57,109.77) -- (526.8,109.77) -- (526.8,207.09) -- (485.62,248.28) -- (238.38,248.28) -- cycle ; \draw   (526.8,109.77) -- (485.62,150.95) -- (238.38,150.95) ; \draw   (485.62,150.95) -- (485.62,248.28) ;
\draw    (367,54) -- (366.38,245.95) ;
\draw    (407.57,204.77) -- (366.38,245.95) ;
\draw    (407.88,108.79) -- (366.69,149.97) ;
\draw    (408.18,12.82) -- (367,54) ;
\draw    (408.18,12.82) -- (407.57,204.77) ;

\end{tikzpicture}
    \end{minipage}
    \begin{minipage}{0.45\textwidth}
        \begin{tikzpicture}[x=0.75pt,y=0.75pt,yscale=-1,xscale=1,scale=0.5]

\draw  [fill={rgb, 255:red, 155; green, 155; blue, 155 }  ,fill opacity=0.5 ] (353.26,55.7) -- (393.76,15.2) -- (522.26,15.2) -- (522.26,109.7) -- (481.76,150.2) -- (353.26,150.2) -- cycle ; \draw   (522.26,15.2) -- (481.76,55.7) -- (353.26,55.7) ; \draw   (481.76,55.7) -- (481.76,150.2) ;
\draw  [fill={rgb, 255:red, 155; green, 155; blue, 155 }  ,fill opacity=0.5 ] (236.46,150.2) -- (276.96,109.7) -- (394.5,109.7) -- (394.5,204.2) -- (354,244.7) -- (236.46,244.7) -- cycle ; \draw   (394.5,109.7) -- (354,150.2) -- (236.46,150.2) ; \draw   (354,150.2) -- (354,244.7) ;
\draw  [fill={rgb, 255:red, 155; green, 155; blue, 155 }  ,fill opacity=0.5 ] (233.19,56.46) -- (273.36,16.28) -- (393.44,16.28) -- (393.44,110.03) -- (353.26,150.2) -- (233.19,150.2) -- cycle ; \draw   (393.44,16.28) -- (353.26,56.46) -- (233.19,56.46) ; \draw   (353.26,56.46) -- (353.26,150.2) ;
\draw   (234.5,55.7) -- (275,15.2) -- (394.5,15.2) -- (394.5,109.7) -- (354,150.2) -- (234.5,150.2) -- cycle ; \draw   (394.5,15.2) -- (354,55.7) -- (234.5,55.7) ; \draw   (354,55.7) -- (354,150.2) ;
\draw  [fill={rgb, 255:red, 74; green, 74; blue, 74 }  ,fill opacity=0.75 ] (194,96.2) -- (234.5,55.7) -- (353.5,55.7) -- (353.5,150.2) -- (313,190.7) -- (194,190.7) -- cycle ; \draw   (353.5,55.7) -- (313,96.2) -- (194,96.2) ; \draw   (313,96.2) -- (313,190.7) ;
\draw   (522.26,109.7) -- (481.76,150.2) -- (353.26,150.2) -- (353.26,55.7) -- (393.76,15.2) -- (522.26,15.2) -- cycle ; \draw   (353.26,150.2) -- (393.76,109.7) -- (522.26,109.7) ; \draw   (393.76,109.7) -- (393.76,15.2) ;
\draw   (522.93,204.59) -- (441.76,285.76) -- (194.93,285.76) -- (194.93,96.37) -- (276.1,15.2) -- (522.93,15.2) -- cycle ; \draw   (194.93,285.76) -- (276.1,204.59) -- (522.93,204.59) ; \draw   (276.1,204.59) -- (276.1,15.2) ;
\draw   (353.5,150.2) -- (313,190.7) -- (194,190.7) -- (194,96.2) -- (234.5,55.7) -- (353.5,55.7) -- cycle ; \draw   (194,190.7) -- (234.5,150.2) -- (353.5,150.2) ; \draw   (234.5,150.2) -- (234.5,55.7) ;
\draw   (193.76,96.44) -- (275,15.2) -- (522.26,15.2) -- (522.26,204.76) -- (441.02,286) -- (193.76,286) -- cycle ; \draw   (522.26,15.2) -- (441.02,96.44) -- (193.76,96.44) ; \draw   (441.02,96.44) -- (441.02,286) ;
\draw   (352.94,56.78) -- (393.44,16.28) -- (521.94,16.28) -- (521.94,110.78) -- (481.44,151.28) -- (352.94,151.28) -- cycle ; \draw   (521.94,16.28) -- (481.44,56.78) -- (352.94,56.78) ; \draw   (481.44,56.78) -- (481.44,151.28) ;
\draw   (395.07,108.94) -- (354.9,149.12) -- (234.82,149.12) -- (234.82,55.38) -- (275,15.2) -- (395.07,15.2) -- cycle ; \draw   (234.82,149.12) -- (275,108.94) -- (395.07,108.94) ; \draw   (275,108.94) -- (275,15.2) ;
\draw   (481.17,244.75) -- (439.67,286) -- (193.75,286) -- (194.87,96.95) -- (236.37,55.7) -- (482.29,55.7) -- cycle ; \draw   (193.75,286) -- (235.24,244.75) -- (481.17,244.75) ; \draw   (235.24,244.75) -- (236.37,55.7) ;

\end{tikzpicture}
\end{minipage}

\caption{Augmentation of an order three tensor. Left: Illustration of the augmentation of a connecting tensor $\bar C_\tau^1$ (light grey left top) with $\wt C_{\tau_1}^1$ in the first dimension (light grey left down) and with $\wt C_{\tau_2}^1$ in the second dimension (light grey right top). Right: Illustration of the augmentation of the connecting tensor from the left in the $0$-dimension (dark grey block). All remaining blocks are set to zero.}
    \label{fig:connecting_tensor_augment}
    
\end{figure}

\subsection{Rank Truncation}\label{subsec:truncation}
The augmented TTN $\wh{X}_{\bar{\tau}}$ usually has doubled ranks in every mode. To keep the computation feasible one has to truncate the ranks of the TTN according to a tolerance $\vartheta$ after each time step. Truncation procedures are well understood such that we omit the detailed description here. A rank truncation algorithm for TTNs can be found in ~\cite[Algorithm 7]{CLS2023}. The error of the rank truncation is under control as the following error bound holds: 
\begin{theorem}[{{\cite[Theorem A.1]{CLS2023} Rank truncation error}}]
\label{thm:err-trunc}
The error of the tree tensor network~$X_{\bar\tau}$, which results from rank truncation of $\wh X_{\bar\tau}$ with tolerance $\vartheta$ according to Algorithm 7 from~\cite{CLS2023}, is bounded by 
$$
\| X_{\bar\tau}-\wh X_{\bar\tau} \| \le c_{\bar\tau}\, \vartheta
\ \ \text{ with }\ 
c_{\bar\tau} =\| C_{\bar\tau} \| (d_{\bar\tau}-1) +1,
$$
where $d_{\bar\tau}$ is the number of nodes in $\bar{\tau}$.
\end{theorem}

\subsection{Comparison to rank-adaptive BUG}
As already noted in the previous section, the parallel BUG reads very similar to the rank-adaptive BUG from~\cite{ceruti2022rank,CLS2023}. However, there are important differences between these two methods. While the differential equations for the leaves $\Phi_l$ are the same, the Galerkin steps $\Psi_\tau$ for the connecting tensor updates are performed in different bases. The parallel BUG uses the basis spanned by all $\bfU_{\tau_i}^0$, while the rank-adaptive BUG uses the augmented basis spanned by all $\wh \bfU_{\tau_i}$.
Note that the ODE for the parallel BUG integrator is of dimension $r\times \cdots \times r$, the ODE for the rank-adaptive BUG integrator is $2r \times \cdots \times 2r$. As we are solely using the old basis in the parallel integrator, all ODEs can be solved fully in parallel. In the rank-adaptive BUG only ODEs on the same level in the tree allow for parallelization. The parallel BUG integrator is less accurate but better suited for large-scale high-performance computations.  

To obtain rank adaptivity with the parallel BUG, we must introduce the augmentation strategy presented in Algorithm~\ref{alg:Augmentation}. The rank-adaptive BUG augments itself on the fly, so no additional augmentation step is needed. However, the recursive augmentation in the $0$-dimension, i.e. computing an orthonormal basis of the range of $(\mat_0(\wh{C}_{\tau_i}^0)^\top,\mat_0(\wh{C}_{\tau_i}^1)^\top)$, is performed for both integrators in the same way. 

\subsection{A fully parallel step rejection strategy for binary trees}
When the solution to an ODE requires a sharp increase of the ranks, both the parallel BUG and the rank-adaptive BUG, fail to capture the dynamics, since they can only double the rank. Therefore, a step rejection strategy must be introduced to allow for arbitrary increases of the ranks at one time step. In this subsection we extend the step-rejection strategy from \cite[section 3.3]{ceruti2023parallel} to binary tree tensor networks.\\
Suppose we have a TTN $Y_{\bar \tau}^0$ with ranks $(r_\tau^0)_{\tau \leq \bar \tau}$ at time $t_0$ and the TTN $\wh Y_{\bar \tau}^1$ with ranks $(\wh r_\tau^1)_{\tau \leq \bar \tau}$ at time $t_1$. For each subtree $\tau \leq \bar \tau$ check the following two conditions: 
\begin{itemize}
    \item[1)] If $\wh r_\tau^1 = 2r_\tau^0$ for some $\tau < \bar \tau$, then the step is rejected. 
    \item[2)] If for some $\tau \leq \bar \tau$ the condition $h \eta_\tau > c \vartheta$ is satisfied (e.g. $c=10$), where  $\eta_\tau = \vert \vert F_\tau(Y_\tau^0) \times_1 \wt \bfU_{\tau_1}^{1,*} \times_2 \wt \bfU_{\tau_2}^{1,*}  \vert \vert $, then the step is rejected. Note that all $\eta_\tau$ can be computed fully in parallel.
\end{itemize}
If a step is rejected, repeat the step with the augmented basis $\wh \bfU_i, i=1,\dots,d$ and augmented connecting tensors $C_\tau^{\text{aug}}$, for $\tau \leq \bar \tau$, where $C_\tau^{\text{aug}}$ equals $C_\tau^0$ augmented with zeros such that its dimensions equal $(\wh r_\tau,\wh r_{\tau_1},\dots,\wh r_{\tau_m})$. Note that this strategy equally applies to the rank-adaptive BUG integrator from \cite{CLS2023}. 
\begin{remark}
    The same strategy applies for non-binary trees, but the number of step rejection constraints per node, i.e., the number of possible combinations of products with two or more $\wt \bfU_{\tau_i}^{1,*}$ factors and $\bfU_{\tau_i}^{0,*}$ as the remaining factors, scales geometrically as $2^{m-1}$.
\end{remark}
\subsection{Robust error bound}
The error bound for Tucker tensors from Theorem \ref{thm:error_bound_tucker} extends to the parallel TTN BUG integrator. Similar to \cite[sec. 5.2]{CLS2023}, a complication arises that we are working on different manifolds at every time step, as the proposed method is rank-adaptive. For a tree $\tau=(\tau_1,\dots,\tau_m)$ we recall the notation $\mathcal{V}_\tau := \C^{r_\tau \times n_{\tau_1} \times \dots \times n_{\tau_m}}$ for the tensor space. The TTN manifold at time step $t_k$ is denoted by $\mathcal{M}_\tau^k := \mathcal{M}_\tau \left((n_\tau)_{\sigma \leq \tau},(r_\sigma^k)_{\sigma \leq  \tau}\right)$. Following the error analysis for the parallel Tucker integrator of Section \ref{sec:parallelTucker}, we make the three assumptions (see also in \cite{CLS2023,CLW2021})
\begin{enumerate}
    \item $F: \mathcal{V}_{\bar \tau} \rightarrow \mathcal{V}_{\bar \tau}$ is Lipschitz continuous and bounded, i.e.,
    \begin{align*}
        \norm{F(Y) - F(\wt Y)} &\leq L\, \norm{Y - \wt Y}  &&\text{for all} \ Y, \wt Y \in \mathcal{V}_{\bar \tau} \\
        \norm{F(Y)} &\leq B \quad &&\text{for all} \ Y \in \mathcal{V}_{\bar \tau}.
    \end{align*}
    \item For $Y\in\mathcal{M}_\tau^k$ near the exact solution $A(t)$ for $t\in [t_k,t_{k+1}]$ we assume, with the
    orthogonal projection $P^k(Y)$ onto the tangent space $\mathcal{T}_Y \mathcal{M}_{\bar \tau}^k$ and the normal projection $P^k(Y)^\perp=I-P^k(Y)$,
    \begin{align*}
        \norm{P^k(Y)^\perp F(Y) } \leq \eps.
    \end{align*}
    \item The error at the initial condition is bounded by $\norm{Y_0 - A(0)} \leq \delta.$
\end{enumerate}
\begin{theorem}\label{thm:error_bound_TTN}
	Under assumptions 1. to 3., the error of the parallel Tucker integrator is bounded by
		\begin{align*}
		\Vert Y_k - A(t_k) \Vert  \le c_1 h + c_2 \varepsilon + c_3 \delta + c_4 k\vartheta,
  \qquad t_0 \le t_k \le T,
		\end{align*}
		where all $c_i$ only depend on the bound and Lipschitz constant of $F$ and on $T$. In particular, the $c_i$ are independent of the singular values of matricizations of connecting tensors.
\end{theorem}
The proof uses similar arguments to the proof of Theorem~\ref{thm:error_bound_tucker} and uses induction over the height of the tree as in Theorem 6.1 in \cite{CLW2021}. 
Therefore, we omit a detailed proof.

\section{Numerical experiments}\label{sec:resultsnum}
We verify the theoretical results by applying the proposed algorithm to a quantum spin system and a problem from radiation transfer. All code is provided online \cite{code}.

\subsection{Long-range interacting quantum system}
Consider an Ising model with long-range interacting particles, where the Schrödinger equation and its corresponding Hamiltonian operator read  
\begin{align}
    i \partial_t \psi  = H \psi, \ \ \text{with} \ \  H = \Omega \sum_{k=1}^d \sigma_x^{(k)} + \Delta \sum_{k=1}^d n^{(k)} + V \sum_{k \neq h}^d \frac{1}{\vert k - h \vert^\alpha} n^{(k)} n^{(h)},\label{eq:long_range_Hamiltonian}
\end{align}
see \cite{saffman2010} for a more detailed description. In this model $\Omega$, $\Delta$, $\alpha \geq 0$ and $V$ are real parameters, $\sigma_x$ the first Pauli matrix, $ n=\begin{pmatrix}
    1 & 0 \\ 0 & 0
\end{pmatrix}$ the projector onto the excited state and $\sigma^{(k)} = (\I \otimes \dots \otimes \I \otimes \sigma \otimes \I \otimes \dots \otimes \I)$ denotes the matrix $\sigma$ acting on the $k$th particle. The first two terms in the Hamiltonian describe a driving term with Rabi frequency $\Omega$ and detuning $\Delta$, e.g. from a laser. The third term describes two particles interacting with each other if and only if both are excited. The parameter $\alpha$ allows to interpolate between different regimes, for example, $\alpha=0$ encodes an all to all interaction. In quantum physics-related settings, values of interest are given by $\alpha = 1$ (Coulomb interaction), $\alpha = 3$ (dipole-dipole interaction) or $\alpha = 6$ (van der Waals interaction) \cite{saffman2010}. \\
For an efficient implementation, we represent the Hamiltonian in tree tensor network format using an HSS (hierarchical semi-separable) decomposition of the corresponding interaction matrices $\boldsymbol{\beta}_{\text{diag}} = \text{diag}(\Omega + \Delta )$ for the diagonal term and $\boldsymbol{\beta}_{\text{int}} = \left(\frac{V}{\vert k-h\vert^\alpha}\right)_{k \neq h}^d$ for the interaction term, as proposed in \cite{ceruti2024_TTNO}. \\
 
We verify the error bound of Theorem \ref{thm:error_bound_TTN} by applying the algorithm to the Schrödinger equation  (\ref{eq:long_range_Hamiltonian}) with strongly and long-range interacting sites, i.e. $\alpha = 1$, where we start from an initial state where all particles are in spin up. The simulations used a balanced binary tree structure and the ODEs in each substep are solved using a fourth-order Runge–Kutta scheme. In the left part of figure~\ref{fig:error_rank_spin_system}, we compare the approximated solution for $d=8$ particles at time $T=1$ with an exact solution. The exact solution is obtained by an exact diagonalization of the Hamiltonian, which can still be done for a system of that size. For comparison, we also include the error of the rank-adaptive BUG integrator for TTNs. We observe that the parallel BUG indeed shows an error of the order $\mathcal{O}(h)$, while the errors of the rank-adaptive BUG behave more like $\mathcal{O}(h^2)$. This coincides with the results in \cite{ceruti2024}. Further, we investigate how the ranks behave over time. The right part of figure~\ref{fig:error_rank_spin_system} shows that the parallel TTN integrator chooses slightly larger maximal ranks compared to the rank-adaptive BUG integrator when choosing the same truncation tolerance $\vartheta$. 



\begin{figure}
    \centering
    \begin{minipage}[t]{0.43\textwidth}
        \includegraphics[scale=0.22]{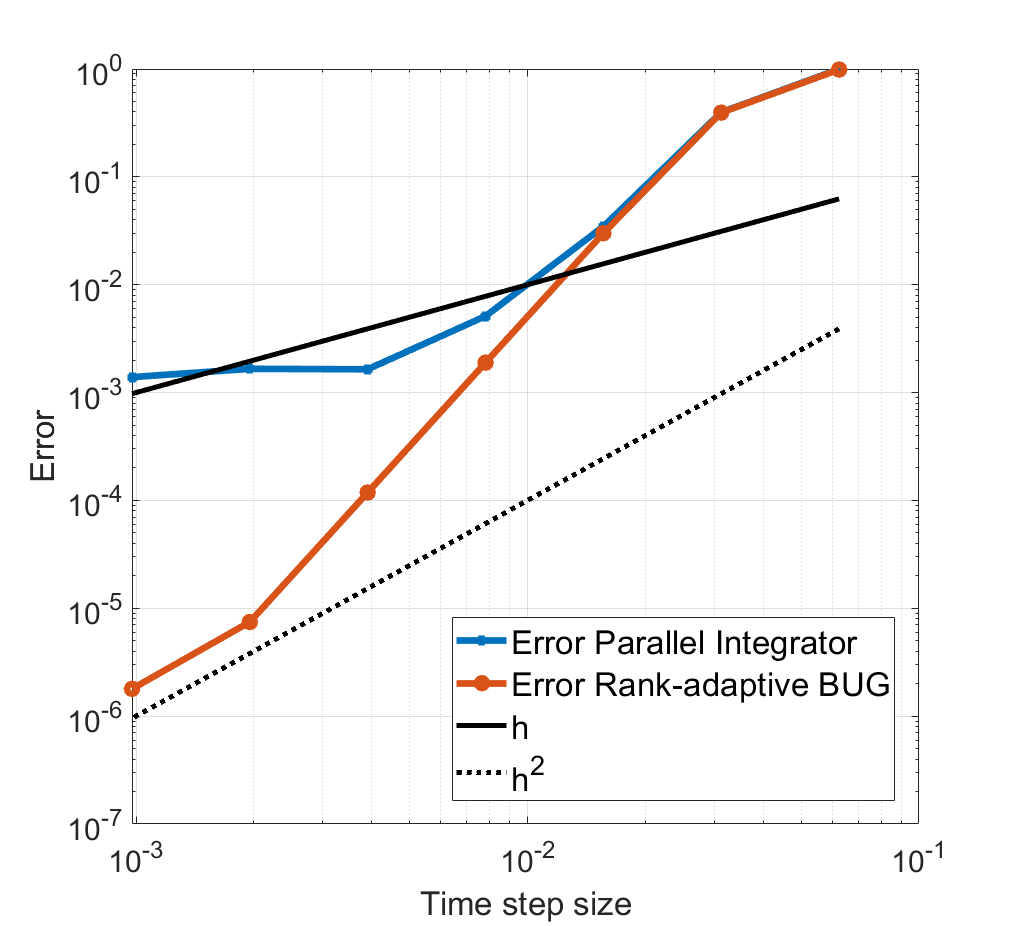}
    \end{minipage}
    \begin{minipage}[t]{0.43\textwidth}
        \includegraphics[scale=0.22]{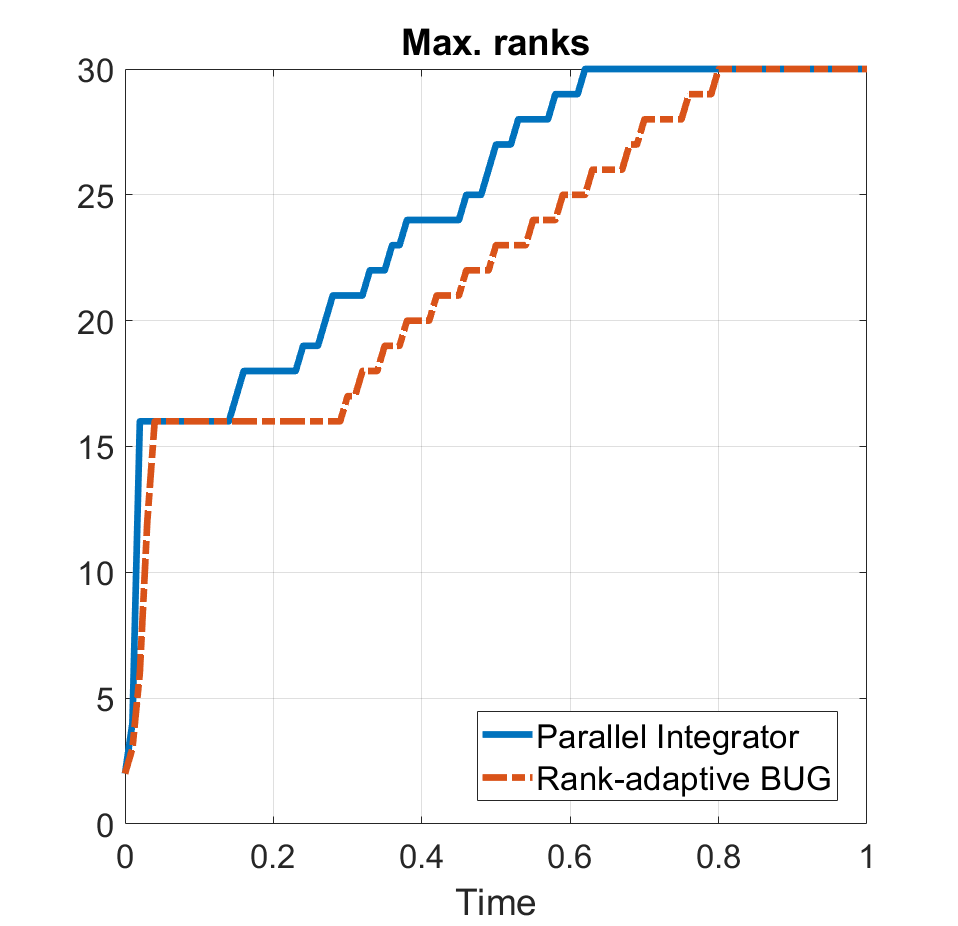}
    \end{minipage}
    \caption{Unspecified parameters are $\Omega=\Delta=V=\alpha=1$, and $\vartheta=10^{-8}$. Left: Errors for $d=8$ particles at time $T=1$. Right: Maximal ranks over time for $d=16$ particles and $r_{\text{max}}=30$.}
    \label{fig:error_rank_spin_system}
\end{figure}


\subsection{Radiative Transfer - Planesource}\label{sec:planesource}
This section investigates the planesource benchmark of radiative transfer \cite{ganapol2008analytical} with an uncertain scattering cross-section. The radiative transfer equation for $f = f(t,x,\mu,\xi,\eta)$ without absorption reads
\begin{align*}
    \partial_t f + \mu\partial_x f = \frac12 \sigma_s\int_{-1}^1f\, d\mu,
\end{align*}
where $\mu\in[-1,1]$ is the projected direction of flight and $\xi,\eta \sim \mathcal{U}(-1,1)$ are uniformly distributed random variables in the interval $[-1,1]$. The scattering cross-section $\sigma_s$ is chosen as $\sigma_s(\xi,\eta) = \sigma_{s,0} + \xi \sigma_{s,\xi} + \eta \sigma_{s,\eta}$ with parameters $\sigma_{s,0} = 5$, $\sigma_{s,\xi} = 4$, and $\sigma_{s,\eta} = 1$. The initial condition is $f(0,x,\mu,\xi,\eta) = \max\{10^{-4}, 1 / \sqrt{2\pi\delta \mathrm{exp}(-x^2/2\delta)}\}$ with $\delta = 0.03^2$. This initial condition resembles the planesource benchmark. Note, however, that the plansource benchmark commonly has deterministic scattering cross-sections. For a detailed discussion of the deterministic planesource benchmark with DLRA, see, e.g., \cite[Section~7.1]{kusch2023stability}. We are now interested in determining the expected value and variance of the scalar flux $\rho(t,x,\xi,\eta) = \int_{-1}^1f\, d\mu$ at time $t_{\mathrm{end}} = 2$. We use a time step size of $h = 0.1 \Delta x$ and use a fourth-order Runge--Kutta scheme. The spatial discretization uses a first-order upwind method with $200$ spatial cells. Moreover, $100$ modal expansion coefficients (moments) are used to discretize the angular domain, and each random variable is discretized using $100$ nodal points using a tensorized Gauss Legendre quadrature. We use a binary tree to represent the solution; that is, a node connects spatial and directional domains while another node connects the uncertain domains. Figure~\ref{fig:Erho} depicts the expected scalar flux, while Figure~\ref{fig:Varrho} depicts the corresponding variance. It is observed that while the parallel integrator is more than twice as fast compared to the rank-adaptive BUG integrator (12 seconds for the parallel integrator compared to 30 seconds for the rank-adaptive BUG integrator), the parallel integrator shows a slightly dampened variance profile. Overall, both methods show satisfactory agreement with the reference collocation solution, which uses a tensorized Gauss Legendre grid in the uncertainty with $100$ collocation points in each dimension. The ranks chosen by the two integrators are shown in Figure~\ref{fig:ranks1D}.

\begin{figure}[h]
    \centering
    \includegraphics[scale=0.4]{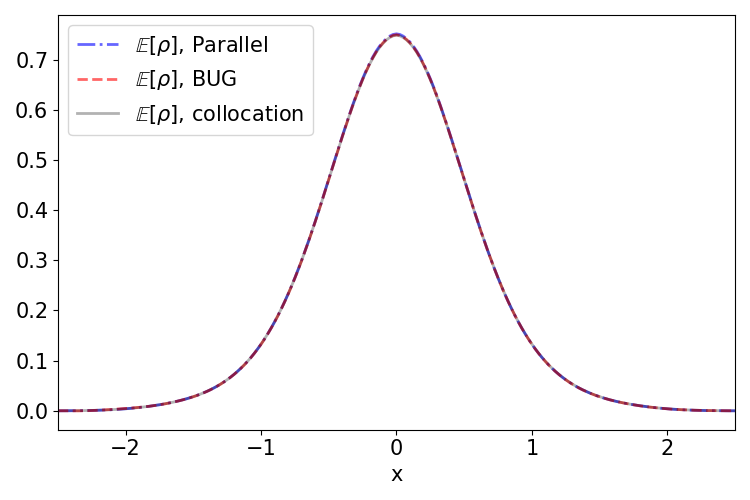}
    \caption{Expected scalar flux at time $t = 2$, using $200$ spatial cells, $100$ moments in direction, and $100$ points in both random variables. The parallel integrator takes 12 seconds, while the rank-adaptive BUG integrator takes 30 seconds.}
    \label{fig:Erho}
\end{figure}

\begin{figure}[h]
    \centering
    \includegraphics[scale=0.4]{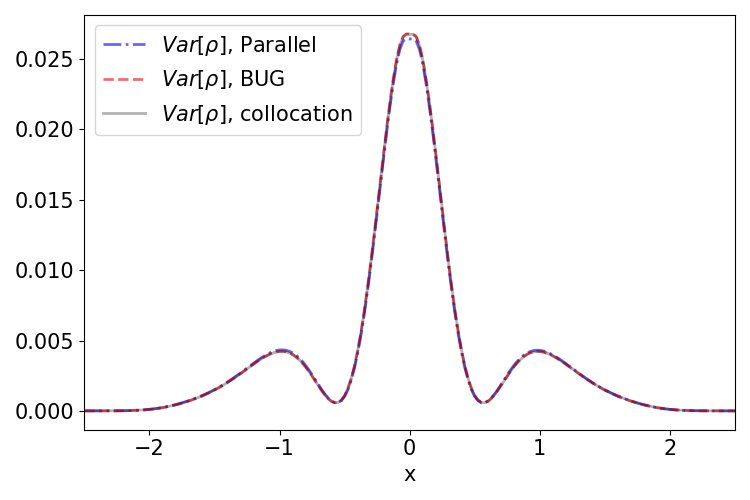}
    \caption{Variance of the scalar flux at time $t = 2$, using $200$ spatial cells, $100$ moments in direction, and $100$ points in both random variables. The parallel integrator takes 12 seconds, while the rank-adaptive BUG integrator takes 30 seconds.}
    \label{fig:Varrho}
\end{figure}

\begin{figure}[h]
    \centering
    \includegraphics[scale=0.45]{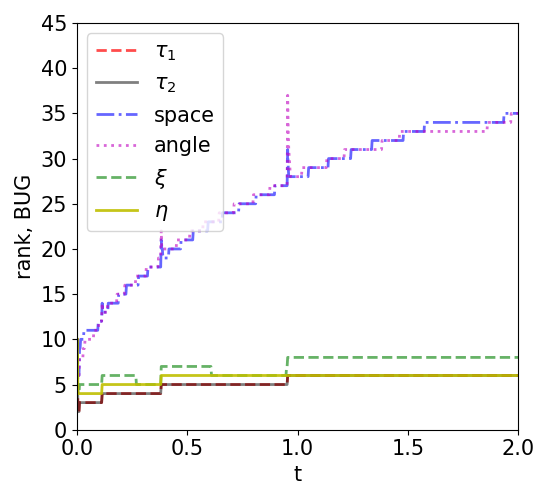}
    \includegraphics[scale=0.45]{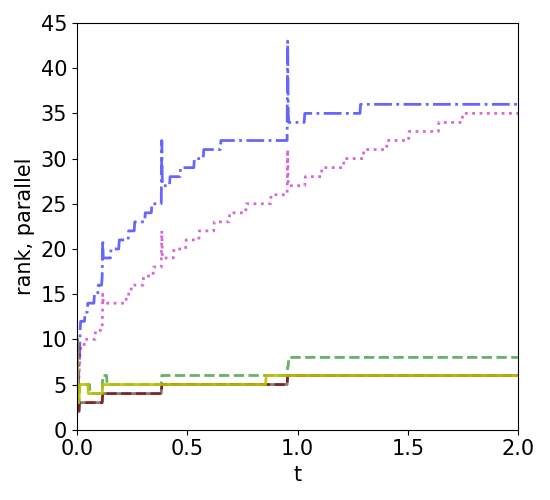}
    \caption{Ranks chosen by the integrator during simulation. The root tensor always has rank $1$. The connecting tensor of the tree $\tau_1$ connects spatial and angular dimensions, whereas the connecting tensor of the tree $\tau_2$ connects the uncertain dimensions.}
    \label{fig:ranks1D}
\end{figure}

\subsection{Radiative Transfer - Linesource}

This section investigates the linesource benchmark with uncertain scattering and absorption cross-sections. The linesource benchmark is a particularly challenging benchmark which is often used to showcase deficiencies of numerical methods \cite{garrett2013comparison}. It solves the two-dimensional radiative transfer equations
\begin{equation}
	\label{eq:rt2d}
	\begin{aligned}
		&\partial_t f + \mathbf{\Omega}\cdot\nabla f + \sigma_t f = \frac{\sigma_s}{4\pi} \int_{\mathbb{S}^2} f \,d\mathbf{\Omega} + \sigma_a f,\qquad 
		  \; \\[1mm]
		& f(t = 0) = \max\left\{10^{-4},\frac{1}{4\pi \sigma^2}\cdot \mathrm{exp}\left(-\frac{\Vert \mathbf{x}\Vert^2}{4\sigma^2}\right)\right\}\,,
	\end{aligned}
\end{equation}
where the scattering and absorption cross-sections are denoted by $\sigma_s$ and $\sigma_a$, respectively. Unlike the standard linesource benchmark, we assume an uncertain scattering cross-section $\sigma_s = 1 + \xi$ where $\xi$ is a random variable that is uniformly distributed in the interval $[-1, 1]$, that is, $\xi \sim \mathcal{U}([-1, 1])$. Similarly, the absorption cross-section is given by $\sigma_a = 1 + \eta$ where $\eta \sim \mathcal{U}([-1, 1])$. As a numerical solver, we use a first-order upwind method with $100^2$ spatial cells. Each uncertain dimension is discretized with $100$ nodal points using a tensorized Gauss Legendre quadrature. The angular domain is discretized with a spherical harmonics (P$_N$) method using $961$ modal expansion coefficients (moments). Similar to the planesource testcase in Section~\ref{sec:planesource}, we use a binary tree to represent the solution.  Figure~\ref{fig:Erho2D} depicts the expected scalar flux, while Figure~\ref{fig:Varrho2D} depicts the corresponding variance. The parallel integrator shows an improved runtime of $72$ seconds compared to $105$ seconds for the rank-adaptive BUG integrator. The numerical results of both integrators agree well. The star-like artifact results from the numerical discretization, which is common in such problems, see, e.g., \cite{ceruti2022rank}. The tolerance has been chosen such that the ranks of the parallel and rank-adaptive BUG integrators are similar, leading to a tolerance of $\vartheta = 2\cdot 10^{-2}$ for the parallel integrator and $\vartheta = 3\cdot 10^{-2}$ for the augmented integrator. The ranks chosen during the simulation can be found in Figure~\ref{fig:ranks2D}.  


\begin{figure}
     \centering
     \begin{subfigure}[b]{0.49\textwidth}
         \centering
         \includegraphics[width=\textwidth]{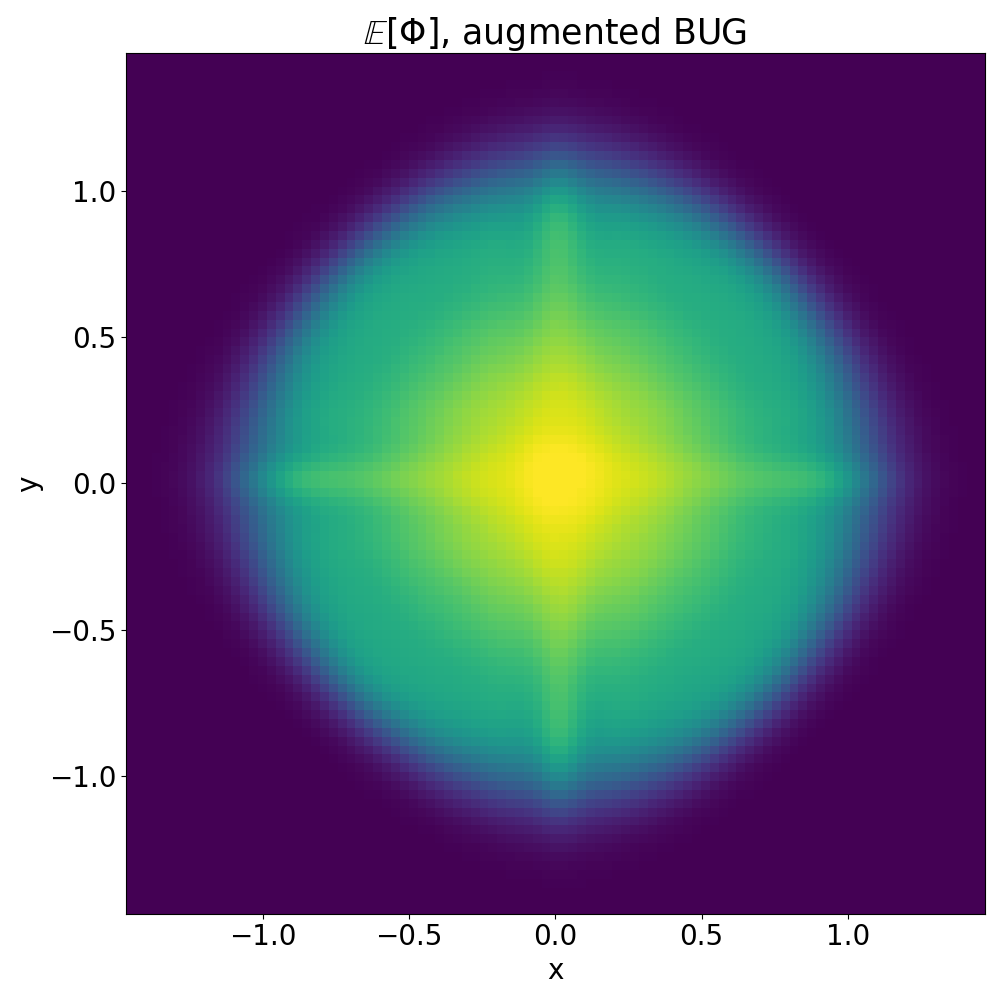}
         \caption{}
     \end{subfigure}
     \hfill
     \begin{subfigure}[b]{0.49\textwidth}
         \centering
         \includegraphics[width=\textwidth]{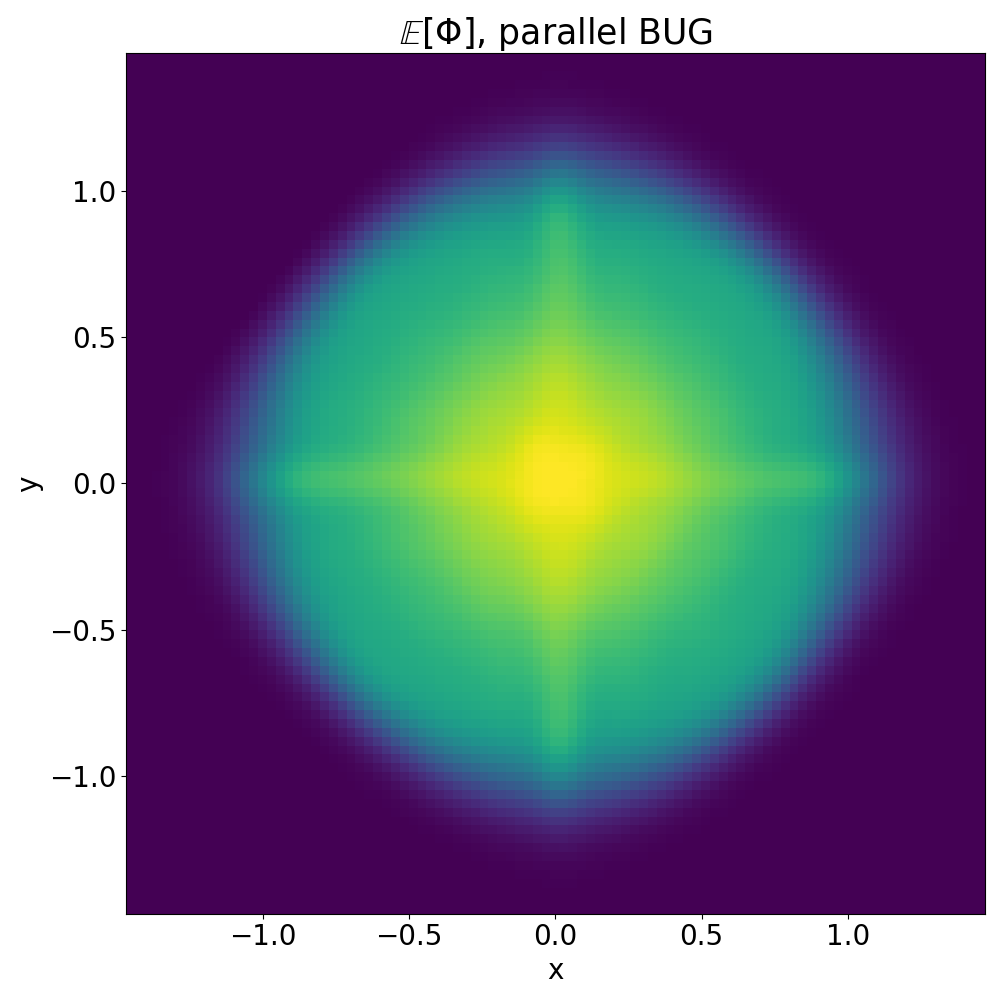}
         \caption{}
     \end{subfigure}
        \caption{
		Expected scalar flux at time $t_{\mathrm{end}} = 1.0$.}
	\label{fig:Erho2D}
\end{figure}

\begin{figure}
     \centering
     \begin{subfigure}[b]{0.49\textwidth}
         \centering
         \includegraphics[width=\textwidth]{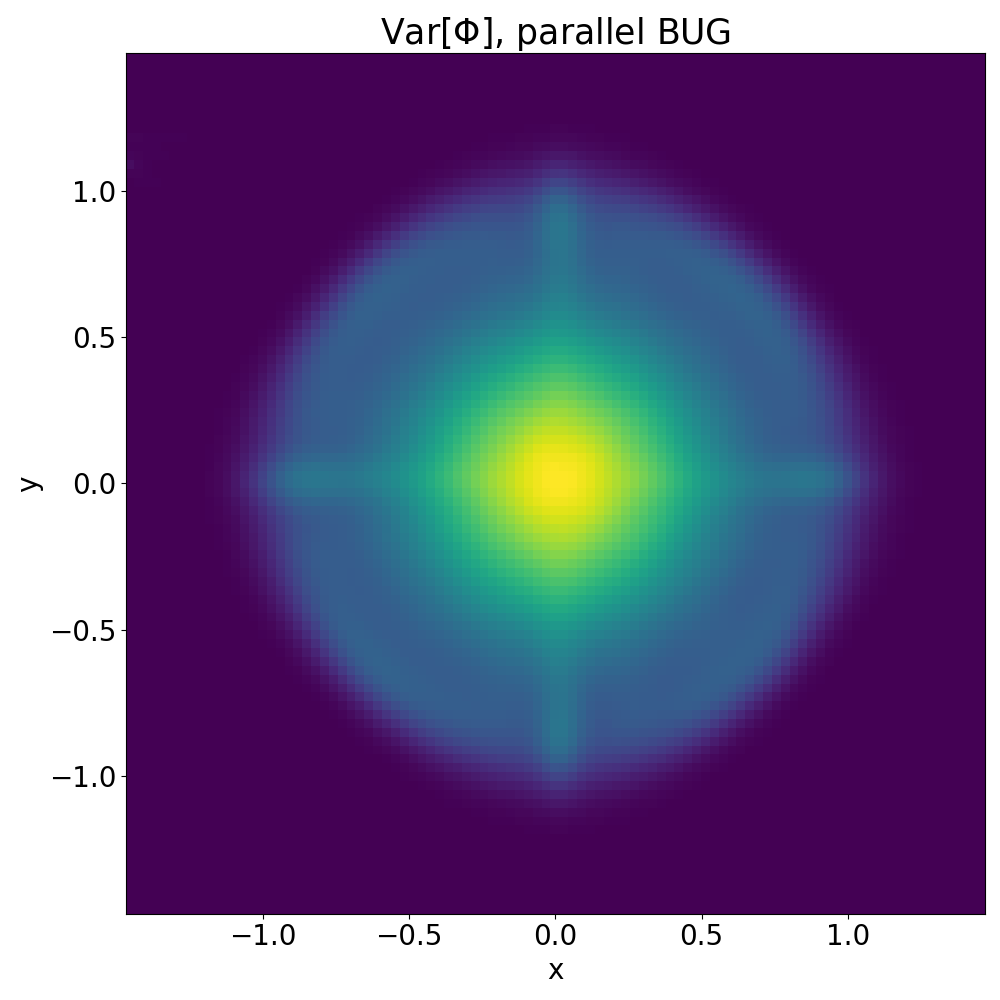}
         \caption{}\label{fig:augmented_BUG_LS}
     \end{subfigure}
     \hfill
     \begin{subfigure}[b]{0.49\textwidth}
         \centering
         \includegraphics[width=\textwidth]{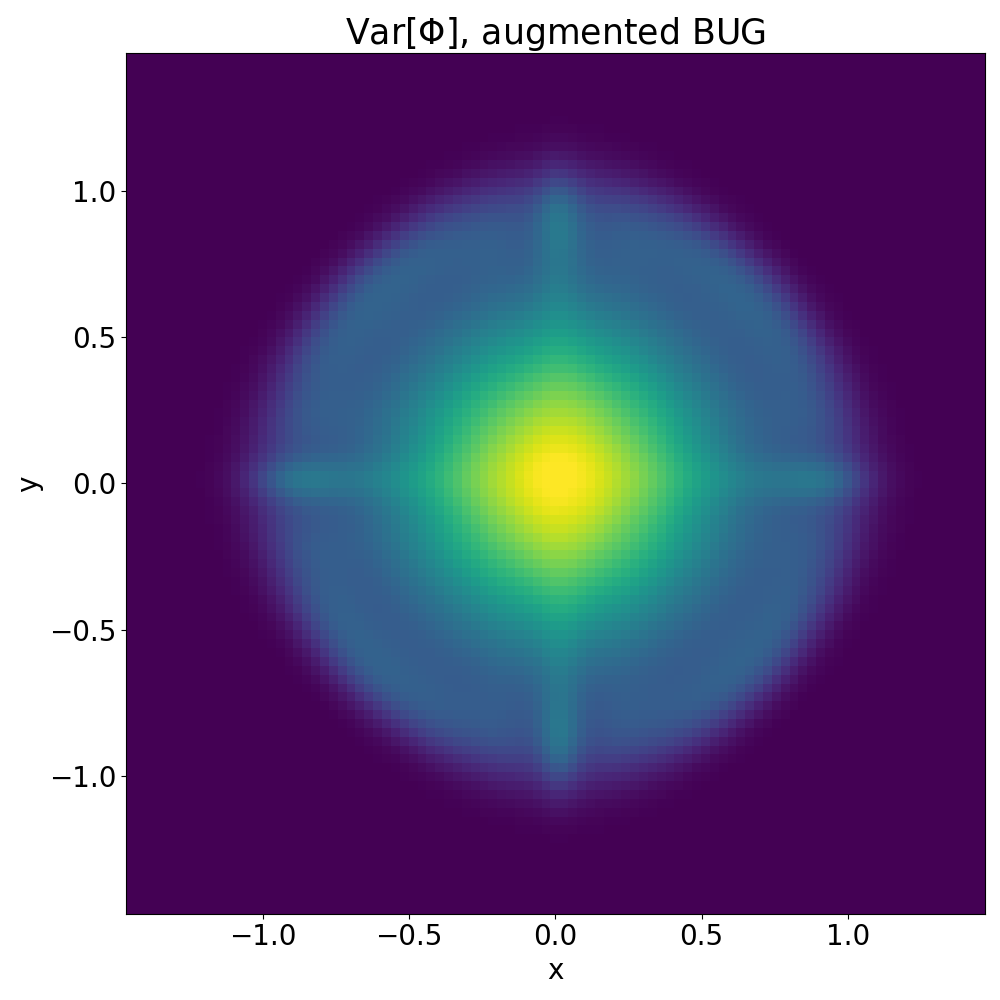}
         \caption{}
         \label{fig:scalar_flux_PN_Lattice_nx350_N21}
     \end{subfigure}
        \caption{Variance of the scalar flux at time $t_{\mathrm{end}} = 1.0$.}
	\label{fig:Varrho2D}
\end{figure}

\begin{figure}[h]
    \centering
    \includegraphics[scale=0.45]{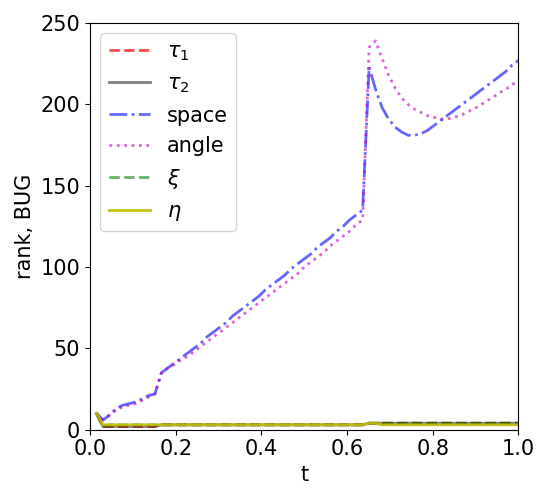}
    \includegraphics[scale=0.45]{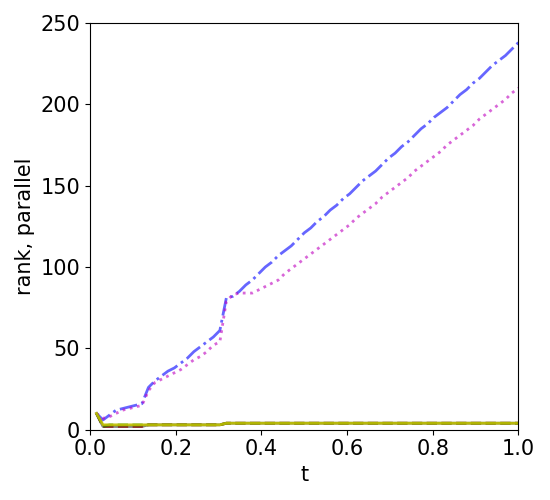}
    \caption{Ranks chosen by the integrator during simulation. The root tensor always has rank $1$. The connecting tensor of the tree $\tau_1$ connects spatial and angular dimensions, whereas the connecting tensor of the tree $\tau_2$ connects the uncertain dimensions.}
    \label{fig:ranks2D}
\end{figure}

\bibliographystyle{siamplain}
\bibliography{main}
\end{document}